\documentclass[12pt, a4paper]{article}
\usepackage{stmaryrd}
\usepackage{setspace}
\usepackage{ps-coil.tex}
\usepackage{amsfonts}
\usepackage[all]{pstricks}
\usepackage{amssymb}
\usepackage{amsthm}
\usepackage{mathrsfs}
\usepackage{latexsym,amsmath}
\usepackage{graphicx}
\usepackage{rotating}
\usepackage[active]{srcltx}
\usepackage[all]{xy}
\usepackage{comment}
\usepackage{mathrsfs}
\usepackage{multirow} 
\newcommand{\w}{\omega}

\renewcommand{\P}{\mathbb{P}}

\newcommand{\dualP}{(\P^2)^\vee}
\newcommand{\sheafG}{\mathcal{G}}
\newcommand{\sheafF}{\mathcal{F}}

\newcommand{\Z}{\mathbb{Z}}
\newcommand{\OO}{\mathcal{O}}
\newcommand{\btimes}{\;\pspicture(0,0)\psdots[dotstyle=+,dotscale=2,dotangle=45](0,0.1)\endpspicture\;}
\newcommand{\A}{A\otimes_Y}

\newcommand{\PP}{\mathbb{P}^1\times\mathbb{P}^1}
\renewcommand{\P}{\mathbb{P}}
\newcommand{\rotsymb}[2]{\begin{rotate}{#2}
	{$#1$}
	\end{rotate}}
\newcommand{\longhookrightarrow}{\xymatrix@C=15pt{\ar@{^{(}->}[r]&}}

\newcommand{\gen}[1]{\langle #1 \rangle}
\newcommand{\Gal}[2]{{\rm Gal}\left(#1/#2 \right)}
\newcommand{\Pic}{{\rm Pic}\;}
\newcommand{\tr}{\rm tr\;}

\newcommand{\Hilb}{{\bf Hilb}\;A}

\newcommand{\EXT}[4]{{\mathcal Ext}^{#1}_{#2}\left(#3,#4\right)}
\newcommand{\ext}[4]{{\rm ext}^{#1}_{#2}\left(#3,#4\right)}
\newcommand{\Ext}[4]{{\rm Ext}^{#1}_{#2}\left(#3,#4\right)}
\newcommand{\HOM}[3]{{\mathcal Hom}_{#1}\left(#2,#3\right)}
\renewcommand{\hom}[3]{{\rm hom}_{#1}\left(#2,#3\right)}
\newcommand{\Hom}[3]{{\rm Hom}_{#1}\left(#2,#3\right)}
\newtheoremstyle{upright}%
        {8pt plus2pt minus4pt}%
        {8pt plus2pt minus4pt}%
	 {\upshape}%
       {}%
       {\bfseries}
        {.}%
        {1em}%
       {}%

 \newtheorem{Lemma}{Lemma}[section]
 \newtheorem{Thm}[Lemma]{Theorem}
 \newtheorem{Cor}[Lemma]{Corollary}
 \newtheorem{Prop}[Lemma]{Proposition}
 \newtheorem{Def}[Lemma]{Definition}

 \theoremstyle{upright}
 \newtheorem{Remark}[Lemma]{Remark}
 
 \newtheorem{Eg}[Lemma]{Example}

 \newtheorem{Construction}[Lemma]{Construction}
 
\title{Line Bundles and Curves on a del Pezzo Order}
\author{Boris Lerner\\  {\footnotesize University of New South Wales}\\
}
\begin{document}
\maketitle
\begin{abstract}
     Orders on surfaces provided a rich source of examples of noncommutative surfaces. In \cite{Hoffmann}
        the authors prove the existence of the analogue of the Picard scheme for orders
        and in \cite{main} the Picard scheme is explicitly computed for an order on $\P^2$ ramified
        on a smooth quartic. In this paper, we continue this line of work, by studying the Picard and Hilbert schemes
        for an order on $\P^2$ ramified on a union of two conics. Our main result is that, upon carefully
        selecting the right Chern classes, the Hilbert scheme is a ruled surface over a genus two curve. Furthermore,
        this genus two curve is, in itself, the Picard scheme of the order.
\end{abstract}
Throughout this paper we assume all objects and maps are defined over an algebraically closed field $k$ of characteristic zero. 
We denote the dimension of any cohomology group over $k$ by the name of the group written
with a non-capital letter for e.g. $\ext iAMN:={\rm dim}_k\Ext iAMN$ and similarly for $h^i$ and ${\rm hom}$.

\section{Introduction}

The study of moduli spaces is an integral part of modern algebraic geometry and representation theory. It is thus very natural,
if one is studying noncommutative surfaces, to wish to understand the various moduli spaces that can be associate to them.
However, even in the commutative case, let alone the noncommutative one, very few examples have been explicitly computed -- which
is what we aim to achieve in this paper.
A rich class on noncommutative surfaces, that has been extensively studied, is that of orders on surfaces, which we now define.

\begin{Def}
        Let $X$ be a normal integral surface. 
        An {\bf order} $A$ on $X$ is a coherent torsion free sheaf of $\OO_X$-algebras such that $k(A):=A\otimes_Xk(X)$ is a central simple
        $k(X)$-algebra. $X$ is called the {\bf centre} of $A$.
\end{Def}
        For example, if $X$ is as above, then 
        any Azumaya algebra on $X$ is an order on $X$. Furthermore,
        it is in fact a {\bf maximal order} in the sense
        that it is not properly included in any other order. For a great reference on orders on surfaces, see \cite{artinjong}
        and \cite{Lectures}.

Since orders are finite over their centres they are in some sense only mildly noncommutative and
many classical geometric techniques can be used to study them. In this paper we first fix an order $A$ on $\P^2$ ramified on a union of two
conics, and study two of its moduli spaces:
\begin{itemize}
        \item [(i)] the moduli space of line bundles on $A$ (see Definition \ref{LineBundle}),
                with a fixed set of Chern classes, denoted by ${\bf Pic}\;A$, and
        \item [(ii)] the moduli space of left quotients of $A$, with a fixed set of Chern classes, denoted by $\Hilb$.
\end{itemize}The first moduli space should be thought of as the Picard scheme of $A$, but one should note that since
$A$-line bundles are only one sided modules, this is not a group scheme. Borrowing terminology from its commutative counterparts, the second
moduli space will be referred to as the Hilbert scheme of $A$ and should be thought of as the space parametrising noncommutative curves on $A$.
Not surprisingly, these two moduli spaces are intrinsically linked; in fact we will prove that $\Hilb$ is a ruled surface over
${\bf Pic}\;A$ and that ${\bf Pic}\; A$ is a genus two curve. Furthermore, by analysing the universal family on $\Hilb$ we will show that 
$\Hilb$ maps to $\dualP \simeq\P^2$ with branch locus
being two conics and their four bitangents.

The inspiration behind this paper comes from \cite{main} where the authors, Chan and Kulkarni, study
the moduli space of line bundles on an order ramified on a smooth quartic. The reader is highly encouraged
to read that paper in order to better understand our motivation. We will explain similarities and differences
between our approaches as we go.

To enable us to begin our project, we use the noncommutative cyclic covering trick,
described in Chapter \ref{nctrick}, to construct our order on $\P^2$.
The key
ingredient to this construction, is a double cover $Y:=\PP\to Z:=\P^2$, 
a line bundle $L\in \Pic Y$ and a morphism
$\phi\!:L_\sigma^{\otimes 2}\to \OO_Y$ where $\sigma$ is the covering involution.
Using this data one constructs a sheaf of algebras $A$ on $Y$ which is an order on $Z$.

The main tool we use for studying $A$-modules is the simple observation that any such module is also naturally
an $\OO_Y$-module. In particular, this allows us to talk about the Chern classes and semistability of $A$-modules
when viewed as $\OO_Y$-modules. Furthermore, we will see that any $A$-line bundle is a rank two vector bundle on $Y$, and so 
their study
is rather different to the study of the Picard scheme of $Y$ and much closer related to the study of rank two 
vector bundles. The main points of difference are that, first of all,
$A$-line bundles do not form a group for they are only left $A$-modules and so their moduli space is not naturally a group 
scheme. Furthermore, the second Chern class, which is zero when one looks at line bundles in the usual setting, plays a crucial role
in their study, as do semistability considerations. More precisely, we are interested in studying
those $A$-line bundles which have minimal second Chern class.

It is certainly not obvious that one can place a bound on the second Chern class of $A$-line bundles and
hence talk about those $A$-line bundles with ``minimal second Chern class''. For Chan and Kulkarni,
this was achieved easily from the fact that for them, $\phi$ was an isomorphism which implied (Proposition 3.8 in \cite{main}) that
any $A$-line bundle was automatically $\mu$-semistable and so by invoking Bogomolov's inequality, this aim was achieved. The authors
used the $\mu$-semistability property further by noting by simply forgetting the extra $A$-module structure, one obtains the map
\[\left\{\txt{moduli space of\\$A$-line bundles\\with minimal $c_2$} \right\}\to\left\{ \txt{moduli space of \\$\mu$-semistable rank \\two vector bundles on $Y$ }\right\}\tag{\textasteriskcentered}\label{map}.\] It is the careful analysis of this map that allowed Chan and Kulkarni to prove that
their moduli space was a genus two curve.

In our case, $\phi$ will not be an isomorphism, and even though we will be able to deduce a lower 
bound for the second Chern class (Proposition
\ref{PropDelta}), the above map of moduli spaces will not be available for us, 
simply because $A$-modules will turn out
to be not $\mu$-semistable in general. Thus  we will use a totally
different approach.

Having bound the second Chern class we will show that it suffices to consider
only two possible first Chern classes: $c_1=\OO_Y(-1,-1)$ with corresponding minimal $c_2=0$ and $c_2=\OO_Y(-2,-2)$ with corresponding
minimal $c_2=2$. The former case will be rather simple and we will prove that the moduli space in that case is just one point. The latter
case will be far more interesting and will be the prime focus of this paper. We will prove, in Theorem \ref{quotientscomputation}, that
for any $A$-line bundle $M$ with this set of Chern classes we have the following exact sequence
\[0\longrightarrow M\longrightarrow A\longrightarrow Q\longrightarrow 0\] where
$Q$ is a quotient of $A$. This establishes a connection between the moduli space of line bundles with  minimal second Chern class
and the Hilbert scheme of $A$ which parametrises quotients of $A$ with specified Chern classes. We will explore this connection in depth and ultimately prove:
\begin{Thm}\label{MainThm}
        Let ${\bf Pic}\;A$ be the moduli space of $A$-line bundles with $c_1=\OO_Y(-2,-2)$ and $c_2=2$ and
$\Hilb$ -- the Hilbert scheme of $A$, parameterising quotients of $A$ with $c_1=\OO_Y(1,1)$ and $c_2=2$. Then
        ${\bf Pic}\;A$ is a smooth genus $2$ curve. ${\Hilb}$ is a smooth ruled surface over ${\bf Pic}\;A$. Furthermore,
        $\Hilb$ exhibits an $8:1$ cover of $\P^2$, ramified on a union of $2$ conics and their $4$ bitangents.
\end{Thm}

In their paper,
Chan and Kulkarni had a remarkably similar result concerning the moduli of line bundles with minimal $c_2$.
They also reduced the study of their moduli space of line bundles
with minimal second Chern class
to two possible first Chern classes. In the first case, the moduli space was a point and in the second case, also 
a genus two curve.
\newline
\\\noindent ACKNOWLEDGEMENTS: This paper is a summary of the author's PhD thesis. Consequently, the author would
like to thank his PhD supervisor Daniel Chan for all his help, patience and inspiration. The author
is also very grateful to Kenneth Chan and Hugo Bowne-Anderson  for all the helpful discussions.
\subsection{Outline of the rest of the paper}
We begin by briefly reviewing the relevant theory of orders on surfaces. 
After this, the rest of the paper is primarily devoted to making sense of, and
proving Theorem \ref{MainThm}. In Section \ref{mod}
we will define and study line bundles with minimal second Chern classes on the order $A$ from Construction \ref{MainConstruction}.
Afterwards, we will introduce the Hilbert scheme
of $A$, which parameterises left sided quotients of $A$. We will
compute its dimension and prove that it is smooth. It is here
that we will also  explore the bizarre covering of $\P^2$ that it exhibits and study its ramification.
In the last section, we will prove that the Hilbert scheme is in fact a ruled surface over
the moduli space. Finally using the map to $\P^2$ we will be able to compute the self 
intersection of the canonical divisor of the Hilbert scheme which will allow us to compute the 
genus of the moduli space.

\section{Preliminaries}
\subsection{Orders on surfaces}

We have already defined the notion of an order on a surface. We will now describe the aforementioned noncommutative
cyclic covering trick which we will later use to construct the order whose moduli spaces we will be studying.
This ``trick'' was introduced by Chan in \cite{ncyclic} and the reader is advised to look there, in 
particular Sections 2 and 3 for
all the relevant details and proofs.
\subsubsection{Noncommutative cyclic covering trick}\label{nctrick}
        The setup is as follows: Let $W$ be a normal integral Cohen-Macaulay scheme
        and $\sigma\in {\rm Aut}\;W$ with $\sigma^e={\rm id}$ for some minimal $e\in \Z^+$. Further,
        assume that $X:=W/\gen\sigma$ is a scheme.
        Given any $L\in \Pic W$, we can form the $\OO_W$-bimodule $L_\sigma$ such that $_{\OO_W}L_\sigma
        \simeq L $ and $(L_{\sigma})_{\OO_W}\simeq \sigma^*L$. Suppose we have an effective
        Cartier divisor  $D$ and an $L\in\Pic W$ such that there exists a non-zero
        map of $\OO_W$-bimodules $\phi\!:L_\sigma^{\otimes e}\stackrel \sim\to \OO_W(-D)\hookrightarrow \OO_W$ 
        satisfying the {\bf overlap condition}; namely that the two maps
        $1\otimes\phi$ and $\phi\otimes 1$ are equal on $L_\sigma\otimes_WL_\sigma^{{\otimes}(e-1)}
        \otimes_WL_\sigma$.
        
        Then 
        \[A=\OO_W\oplus L_\sigma\oplus\cdots\oplus L_\sigma^{\otimes (e-1)}\]
        is an order on $X$ with multiplication given by:
        \[L_\sigma^{i}\otimes_{W} L_\sigma^{j}\longrightarrow \left\{  
        \begin{array}[]{ll}
                L_\sigma^{\otimes(i+j)},&i+j<e\\
                L_\sigma^{\otimes(i+j)}\stackrel{1\otimes\phi\otimes 1}{\longrightarrow}L_\sigma^{\otimes\left( i+j-e \right)},&i+j\geq e
        \end{array}
        \right.\] which is independent of any choice that needs to be made when applying the map $1\otimes\phi\otimes 1$
        due to the overlap condition. Orders constructed in this manner are called {\bf cyclic orders}. We
        will almost always regard $A$ as an $\OO_W$-bimodule on $W$, in which case we pay special
        consideration to the fact that it is not $\OO_W$-central.

        Note that if we want to use this method to construct an order on a specific scheme $X$ we also need a way of finding a scheme
$W$ and an automorphism $\sigma\in {\rm Aut}\;W$ such that $W/\gen\sigma=X$. We can do so, using the classical cyclic covering
construction.
\begin{Construction}\label{cons}
       Let $X$ be a normal integral scheme, let $E\geq 0$ be an effective divisor and $N\in \Pic X$ such that
       $N^{\otimes e}\simeq \OO_X(-E)$. Then \[\pi\!:W:=Spec_X(\OO_X\oplus N\oplus\cdots\oplus N^{\otimes(e-1)})\to X\] is a cyclic cover of $X$.
       See Chapter 1, Section 17 of \cite{Barth} for more details. 
      Note that if $\sigma$ is the generator of $\Gal WX$ then $W/\gen\sigma=X$. To construct an
      order on $X$ using the noncommutative cyclic covering trick, let  $E'\geq 0$
      be another effective divisor on $X$ and let $D=\pi^* E'$. Find an $L\in \Pic W$ and a non-zero morphism 
      (if one exists) $\phi\!:L_\sigma^{\otimes e}\to\OO_W(-D)$ satisfying the overlap condition. Then as described
      above, we can construct an order on $X$ which we will denote by $A(W/X;\sigma,L,\phi)$. This
      order is ramified on $E\cup E'$, see \cite{ncyclic} Theorem 3.6 for a proof of this. We suppress 
      $E,E'$ and $D$ from the notation.
      \end{Construction}
      \subsection{The order we wish to study}\label{orderA}
In this section we will use the noncommutative cyclic covering trick to construct a del Pezzo order on $\P^2$ ramified
on a union of two conics. It is the moduli space and Hilbert scheme of this order that we will be investigating
for the remainder of this paper.

\begin{Construction}\label{MainConstruction}
        Let $Z=\P^2$ and $\pi\!:Y\to Z$ be a double cover ramified on a smooth conic $E\subset Z$ and let
        $\sigma$ be the covering involution. It is well known that $Y\simeq \PP$, $\Pic Y=\Z\oplus \Z$ 
        and that $\sigma^*(\OO_Y(n,m))=\OO_Y(m,n)$. Let $H$ be the inverse image of a general line in $Z$. It is a $(1,1)$-divisor
        and is ample.
        Let $E'\subset Z$ be a second smooth conic, intersecting $E$ in $4$ distinct point, let $D=\pi^*E'$ which is
        a smooth $(2,2)$-divisor, let $L=\OO_Y(-1,-1)\in \Pic Y$ and fix once and for all a morphism $\phi\!:L_\sigma^{\otimes 2}\stackrel{\sim}
        {\longrightarrow} \OO_Y(-D)\hookrightarrow \OO_Y$.  Any such $\phi$ satisfies the overlap condition
        and so $A:=A(Y/Z;\sigma,L,\phi)$ is a maximal
        order, in a division ring, on $Z$ ramified on $E\cup E'$. See \cite{Thesis} Chapter 1 for full proofs.        
\end{Construction} 
As mentioned previously, in \cite{main} the authors also consider a maximal order on $\P^2$ ramified, in their case, on a smooth quartic.
More importantly, the relation used in their construction was of the form $L_\sigma^{\otimes 2}\simeq \OO_W$. As we
shall see this small difference makes their techniques for the study of ${\bf Pic}\;A$, unusable in our case.

\subsection{The canonical bimodule}
To finish off the introduction we would like to explain in what sense our order $A$ is del Pezzo. 
We begin with the definition of the canonical bimodule which is the analogue of the canonical sheaf 
on a scheme. 

\begin{Def}\label{delPezzodef}
       Let $X$ be a normal integral scheme and $A$ an order on $X$.
        The canonical bimodule of $A$ is defined to be 
        \[\omega_A:=\HOM{\OO_X}{A}{\omega_X}.\] 
\end{Def} Mimicking the commutative definition,
        we say that $A$ is del Pezzo if $\omega_A^*:=\HOM A{\omega_X}A$ is ample.
For more details see \cite{bimod} Section 3.

If $X$ is Gorenstein, then $\omega_A=\HOM {\OO_X}A{\OO_X}\otimes_{X}\omega_X$. Using the reduced trace map, we can identify 
$\HOM {\OO_X}A{\OO_X}$ as an $A$-subbimodule of $k(A)$ and so $\w_A$ can be identified as an $A$-subbimodule of $k(A)\otimes_X\omega_X$.
The next theorem allows us to determine, in the case where $A$ is constructed using Construction \ref{cons}, precisely
what this subbimodule is. Knowledge of $\omega_A$ will be very valuable to us
in the future for various homological computations.

\begin{Thm}\label{canonical}
	
        Let $X$ be a normal integral Gorenstein scheme.
        Let $A:=A(W/X;\sigma,L,\phi)$ be an order on $X$ as described in Construction \ref{cons}
        and let $R\subset W$ be the reduced pullback of $E$ to $W$.

        Then 
        \begin{align*}
        \omega_A
        &=A\otimes_{W}L_\sigma\otimes_{W}\OO_W\left( (e-1)R+D \right)\otimes_{X}\omega_X\\
        &=A\otimes_{W}L_\sigma\otimes_{W}\OO_W(D)\otimes_{W} \omega_W
        \end{align*}
\end{Thm}in $k(A)\otimes_W\w_W$.
\begin{proof}
               From Lemma 17.1 of
                \cite{Barth}  and the adjunction formula we know that \[\omega_W=\pi^*\omega_X\otimes_W
                \HOM{\OO_X}{\OO_W}{\OO_X}=\pi^*\omega_X\otimes_W\OO_W(\left( e-1 \right)R).\]
                Thus, using the reduced trace map we have:
\begin{align*}
	\HOM{\OO_X}A{\OO_X}&=\left\{ f\in k(A)\;|\;{\tr}(fA)\subseteq \OO_X \right\}\subseteq k(A)\\
	&=D_0\oplus D_1\oplus\cdots\oplus D_{e-1}
\end{align*}where
\[	D_{i-1}=\left\{ f\in L_\sigma^{\otimes(i-1)}\otimes_{W}k(W)\;|\;{\tr}(fL_\sigma)\subseteq 
	\OO_X \right\}=L_\sigma^{\otimes(i-1)}\otimes_{W}\OO_W\left( (e-1)R +D\right)\]
for $1\leq i\leq e$, and so: \[\HOM{\OO_X}A{\OO_X}=A\otimes_{W}L_\sigma\otimes_{W}\OO_W\left( (e-1)R+D \right).\]Thus:
\begin{align*}
	\omega_A&:=\HOM{\OO_X}A{\omega_X}\\
	&=A\otimes_{W}L_\sigma\otimes_{W}\OO_W\left( (e-1)R+D \right)\otimes_{W}\pi^*\omega_X\\
	&=A\otimes_{W}L_\sigma\otimes_{W}\OO_W(D)\otimes_{W} \omega_W.
\end{align*}
\end{proof}
Applying this theorem to our specific order $A$  we get:
\begin{Prop}
        Let $A:=A(Y/Z;\sigma,L,\phi)$ be as in Construction \ref{MainConstruction}. Then $\omega_A=A\otimes_Y\OO_Y(-H)$. In
        particular, $A$ is del Pezzo.
\end{Prop}
\begin{proof}        
We simply apply Theorem \ref{canonical} and use the well known fact that $\omega_Y=\OO_Y(-2,-2).$
\end{proof}

From now on, unless explicitly stated otherwise, $A$ denotes $A(Y/Z;\sigma,L,\phi)$ -- the order constructed in 
Construction \ref{MainConstruction}.

\section{The Moduli Space of $A$-Line Bundles}\label{mod}
In this section we will study line bundles on $A$.
\begin{Def}\label{LineBundle}
        Let $X$ be a normal integral scheme and $B$ an order in a division ring $k(B)$ on $X$.
        Let $M$ be a sheaf of left $B$-modules. We say $M$ is a {\bf line bundle} on $B$ if $M$ is locally projective
        as a $B$-module and ${\rm dim}_{k(B)}k(B)\otimes _B M=1$. The set (not group) of isomorphism classes of
        $B$-line bundles will be denoted by $\Pic B$.
\end{Def} 

The following proposition gives a very useful criterion for checking whether an $A$-module is in fact an
$A$-line bundle.

\begin{Prop}\label{PropLocFree}
        $M\in \Pic A$ if and only if $M$ is an $A$-module such that $_YM$ is a rank two locally free sheaf on $Y$.
\end{Prop}
\begin{proof}
        Follows from the fact that, $A$ is locally of global dimension $2$. See Proposition 2.02 in \cite{Thesis}
        for a full proof.
\end{proof}
\begin{Eg}\label{AtensorN}
        Suppose $N\in \Pic Y$. Then $\A N$ is an $A$-line bundle since it is clearly
        an $A$-module and is locally free of rank two over $Y$.
\end{Eg}

\subsection{Chern classes of $A$-line bundles}\label{Chernclasses}
In this section we study the possible Chern classes of line bundles on $A$.
Recall that whenever we speak of Chern classes for any $M\in \Pic A$ we imply that
we are talking about the $\OO_Y$-module $_YM$.

The first natural question to ask about any $A$-line bundle is what could be its 
first Chern class. We answer this in the following proposition.
As it turns out, the possibilities are fairly limited.
\begin{Prop}
        Let $M\in \Pic A$. Then $c_1(M)=\OO_Y(n,n)$ for some $n\in \Z$. Conversely, given any such $n$,
        $\A\OO_Y(0,n+1)\in\Pic A$ with $c_1=\OO_Y(n,n)$.
\end{Prop}
\begin{proof}
        	First note that we have a chain of $\OO_Y$-submodules 
                $M(-D):=L_\sigma^{\otimes 2}\otimes_Y M<L_\sigma\otimes_Y M<M$ which means
	\[0\to\frac{L_\sigma\otimes_YM}{L_\sigma^{\otimes 2}\otimes_YM}\to\frac{M}{M(-D)}\to\frac{M}{L_\sigma\otimes_YM}\to 0\]
	is an exact sequence. Let $Q=M/(L_\sigma\otimes_YM)$. The above then becomes:
	\[0\to L_\sigma\otimes_Y Q\to M|_D\to Q\to 0.\] Now $M|_D$ is a locally free sheaf on $D$ of rank 2, and so $L_\sigma\otimes
	_YQ$ and hence $Q$ must be line bundles on $D$. 
        Consequently, $c_1(Q)=D$ and so
        $c_1(M)=c_1(L_\sigma\otimes M)+D$. Hence $c_1(M)=\sigma^{*}c_1(M)$ and the result follows.

        To see the converse, first note that by Example \ref{AtensorN} we know that $M:=\A\OO_Y(0,n+1)$ is indeed an $A$-line bundle.
        Furthermore, $c_1(M)=c_1(\OO_Y(0,n+1)\oplus \OO_Y(n,-1))=\OO_Y(n,n)$.
\end{proof}

Having classified all the possible first Chern classes of $A$-line bundles, we move on to
see what can be said about the second Chern class. As we shall see, the second Chern class
has a strict lower bound analogous to Bogomolov's inequality, which we now recall.

Let $X$ be a smooth projective surface and $\sheafF$ a torsion free coherent sheaf on $X$ with
Chern classes $c_1$, $c_2$ and rank $r$. Fix an ample divisor $H$ on $X$. The {\bf gradient} of $\sheafF$ is defined to be
\[\mu(\sheafF):=\frac{c_1.H}{r}.\] $\sheafF$ is said to be $\mu$-semistable if for any subsheaf $\sheafF'\subset\sheafF$ we have
$\mu(\sheafF')\leq \mu(\sheafF)$. Bogomolov's inequality (Theorem 12.1.1 in \cite{Potier}) states, 
that if $\sheafF$ is semistable then \[
\Delta(\sheafF):=4c_2-c_1^2\geq 0.\] Thus, if
considering any class of semistable sheaves on $X$ with a fixed first Chern class, the second Chern class is bounded from below. 

In \cite{main} the authors were able to show to that for their cyclic order $A$, any $A$-line
bundle was automatically $\mu$-semistable as sheaf on $Y$ and 
could thus bound the second Chern class using Bogomolov's inequality. 

We modify their proof and achieve a slightly weaker result for
our order.

\begin{Prop}\label{MyBogomolov}
        Let $M\in \Pic A$  and let $N<M$ be an $\OO_Y$-subsheaf. Then
        \[\mu(N)\leq \mu(M)+1\tag{1}.\label{ahs}\] 
\end{Prop}
\begin{proof} Note that $M$ is locally free of rank $2$ over $Y$. Thus the result is 
               clear if rank $N$= 2 and so we assume rank $N=1$.
        Observe that $c_1(L_\sigma\otimes N).H=c_1(L).H+\sigma^*c_1(N).H=
	-2+c_1(N).H$. Now
	\begin{align*}
		\mu(N\oplus L_\sigma\otimes_YN)&=\frac{c_1(N).H+c_1(L_\sigma\otimes_YN).H}{2}\\
		&=\frac{2c_1(N).H-2}{2}\\
		&=c_1(N).H-1\\
		&=\mu(N)-1
	\end{align*} and so $\mu(N)=\mu(N\oplus L_\sigma\otimes_YN)+1\leq \mu(M)+1$.
\end{proof}
It is easy to see that this inequality is tight. For example the $A$-line bundle $A$ has gradient
$\mu(A)=-1$ and an $\OO_Y$-submodule $\OO_Y<A$ with $\mu(\OO_Y)=0$. Thus $A$-line bundles are in general not
$\mu$-semistable and so we can not apply Bogomolov's inequality to give a lower bound for the second Chern class.
Furthermore, as mentioned earlier, this implies we can not use the map (\textasteriskcentered) from
page \pageref{map} in order to study the moduli space of line bundles, simply because this map does not
exist for us.

Luckily, due to a deep theorem by Langer in \cite{Langer} the result of Proposition \ref{MyBogomolov} is good enough
to achieve a lower bound on $c_2$.
\begin{Prop}\label{PropDelta}
        Let $M\in \Pic A$ with Chern classes $c_1$ and $c_2$. Then
        \[\Delta(M)=4c_2-c_1^2\geq -2.\]
\end{Prop}
\begin{proof}
        Follows immediately from Theorem 5.1 of \cite{Langer} with $D_1=H$ and Proposition \ref{MyBogomolov}.
\end{proof}
\begin{Remark}
        The above theorem can also be proven using rather elementary techniques, without needing
        the generality of \cite{Langer}.
\end{Remark}

In Section \ref{Chernproof} we will prove that $4c_2-c_1^2\geq -2$ is in fact a sufficient condition
to guarantee that there exists am $A$-line bundle with these Chern classes.

Having shown that for a fixed first Chern class, the second Chern class of any $A$-line bundle is bounded from
below, we begin studying those line bundles, with minimal second second Chern class. In particular,
we would like to determine what the moduli space of such bundles is. 

The existence of a projective coarse moduli scheme parametrising $A$-line bundles with minimal $c_2$ follows
easily from Theorem 2.4 in \cite{Hoffmann} and Proposition \ref{PropDelta}. In fact
this moduli space is smooth because $A$ is del Pezzo.
For a full explanation and proof, see \cite{Thesis} Chapter 2.
\begin{Remark}\label{RemarkOnC}
        Since the functor $\OO_Y(nH)\otimes_Y-$ is a category autoequivalence of $A$-Mod, it induces an automorphism of
        the moduli space of $A$. Note that for any $M\in \Pic A$, $c_1(\OO_Y(nH)\otimes_YM)=2nH+c_1(M)$. Since by the
        previous proposition, $c_1(M)=mH$ for some $m\in \Z$ we me may assume that $c_1(M)=\OO_Y(-1,-1)$ or $c_1(M)=
        \OO_Y(-2,-2)$. 
\end{Remark}

Before we begin our analysis of $A$-line bundles with minimal $c_2$, we need
to examine the inequality (\ref{ahs}) we met in Proposition \ref{MyBogomolov} a little further.
\begin{Def}
        Let $X$ be a  surface and $V$ a vector bundle on $X$. 
        We say $V$ is {\bf almost semistable} if for any subbundle $V'\subset V$,
        $\mu(V')\leq \mu(V)+1$.
\end{Def}
\begin{Prop}\label{ahsprop}
        Let $X$ be a surface and $V$ a vector bundle on $X$.
        \begin{enumerate}
                \item $V$ is almost semistable if and only if $V\otimes_{X} N$ is
                        almost semistable for 
                        all $ N\in \Pic X$.
                \item If $V$ is rank $2$ and almost semistable, then so is $V^*$.
        \end{enumerate}
\end{Prop}
\begin{proof}$\text{}$

        \begin{enumerate}
                \item
                        Suppose $V$ is almost semistable and  $V'\subseteq V\otimes_X N$. Then
                        $V'\otimes_X N^{-1}\subseteq V$ and so
                        $ \mu(V'\otimes_X N^{-1})\leq \mu(V)+1$ thus $\mu(V')-c_1(N).H\leq \mu(V)+1$
                        and so $\mu(V')\leq (V\otimes_X N)+1$.
	To see the converse simply let $ N=\OO_X$.
        \item Follows from (1) and the fact that $V^*\simeq V\otimes_X({\rm det}\;V)^{-1}$.
        \end{enumerate}
\end{proof} 
As we have seen in Proposition \ref{MyBogomolov}, $A$-line bundles are almost semistable. 
We will use the above proposition later on for proving various
properties regarding line bundles on $A$.
\subsection{Case 1: $c_1=\OO_Y(-1,-1)$}
As mentioned in Remark \ref{RemarkOnC} the problem of studying the moduli space of $A$-line bundles
with minimal $c_2$ naturally breaks up into two parts $c_1=\OO_Y(-1,-1)$ or $\OO_Y(-2,-2)$. In this subsection
we examine the former case.
By Proposition \ref{PropDelta} the minimal $c_2=0$ and
this corresponds to $\Delta =-2$, the smallest value possible. It is easy to see that the moduli space
of $A$-line bundles with these Chern classes isn't empty for clearly $A$ itself, regarded as a left $A$-module,
has the desired Chern classes. As it turns out, this is in fact the only such $A$-line bundle.
\begin{Thm}\label{Thm0}
        Let $M\in \Pic A$ with $c_1=\OO_Y(-1,-1)$ and $c_2=0$. Then $M\simeq A$. In particular, the coarse
        moduli space of $A$-line bundles with these Chern classes is  a point.
\end{Thm}
\begin{proof}
        By the Riemann-Roch theorem
        $\chi(M)=1>0$. On the other hand $h^2(M)=h^0(\omega_Y\otimes_Y M^*)$ and $c_1(\omega_Y\otimes_YM^*)=
        \OO_Y(-3,-3)$. As we saw in Proposition \ref{MyBogomolov}, $M$ is almost semistable, and so by
        Proposition \ref{ahsprop}, $\omega_Y\otimes_Y M^*$ is also almost semistable and so
        $h^2(M)=0$. Thus $h^0(M)\neq 0$ and so $\OO_Y\hookrightarrow
	M$ which gives an injection of $A$-modules $A\otimes_Y \OO_Y=A\hookrightarrow M$. Since
	their first Chern classes equal, the map must be an isomorphism.

        Finally \[\Ext1AAA=\Ext1Y{\OO_Y}{\OO_Y\oplus\OO_Y(-1,-1)}=H^1(Y,\OO_Y\oplus\OO_Y(-1,-1))=0.\]  where
        the first equality follows from Proposition 2.6 of \cite{main} which asserts that
        there is a natural isomorphism of functors $\Ext iA{A\otimes_YN}-\simeq \Ext iYN-$ for any
        $N\in \Pic Y$. See Chapter 3 Exercise 5.6 of \cite{Hartshorne} for
                        the cohomology of $\PP$.  Thus
        the tangent space at the point corresponding to the $A$-line bundle $A$ is $0$-dimensional 
        and so the moduli space is just a point.
\end{proof} 
\subsection{Case 2: $c_1=\OO_Y(-2,-2)$}
We now study the second case mentioned in Remark \ref{RemarkOnC}: the case
where $c_1=\OO_Y(-2,-2)$. By Proposition \ref{PropDelta} the minimal $c_2=2$ which corresponds
to $\Delta=0$ which is its second smallest value for clearly $\Delta$ must be even. Note that $A\otimes_Y\OO_Y(-1,0)$ is an $A$-line bundle by Example \ref{AtensorN}
and has the desired Chern classes. Thus the moduli space of such $A$-line bundles is not empty.

From now on  ${\bf Pic}\;A$ will denote the moduli space of $A$-line bundles
with $c_1=\OO_Y(-2,-2)$ and $c_2=2$. We first establish all the possible $\OO_Y$-module structures
that such $A$-line bundles can have.
\begin{Thm}[$\OO_Y$-module structure]\label{ystruc}
        Let $M\in\Pic A$ with $c_1=\OO_Y(-2,-2)$ and $c_2=2$. Then either $M\simeq 
        \OO_Y(-1,-1)\oplus\OO_Y(-1,-1)$ as and $\OO_Y$-module or $M\simeq A\otimes_Y(-F)$ as
        $A$-modules where $F$ is either a $(1,0)$ or a $(0,1)$-divisor.
\end{Thm}
\begin{proof}
        The beginning of this proof is very similar
        to the proof of Theorem \ref{Thm0} so we skip some details which we have already explained there.
        Let $M_1=\OO_Y(1,1)\otimes_YM$. Then $c_1(M_1)=2c_1(\OO_Y(1,1))+c_1(M)=0$ and $c_2(M_1)=
        c_2(M)+c_1(M).c_1(\OO_Y(1,1))+c_1(\OO_Y(1,1))^2=2-4+2=0$.
        Thus by the Riemann-Roch theorem
        $\chi(M_1)=2>0$. 
         However, $h^2(M_1)=h^0(\omega_Y\otimes_Y M_1^*)$ whilst $\omega_Y\otimes_YM_1^*$ is almost semistable
         with gradient $-4$ and so 
        $h^0(\omega_Y\otimes_YM_1^*)=h^2(M_1)=0$ and so $h^0(M_1)\neq 0$. 
        Thus we 
        know $\OO_Y(-1,-1)\hookrightarrow M$. 
        Now if there exists a bigger $\OO_Y$-line bundle (ordered by inclusion) which embeds into $M$ then 
        $\OO_Y(-F)$ 
        embeds into $M$ where $F$ is either a $(1,0)$ or a $(0,1)$-divisor. This extends to an embedding
        $A\otimes_Y\OO_Y(-F)\hookrightarrow M$ of $A$-line bundles 
        and so comparison of the first Chern classes guarantees that $M\simeq\A\OO_Y(-F)$. 
        Suppose on the other hand that
        $\OO_Y(-1,-1)$ is the biggest line bundle which embeds into $M$. Let the quotient be $Q$ and
        note that it is torsion free.
        By Proposition 5 (ii) in \cite{Friedman} $Q=L'\otimes_Y{\mathcal I}_Z$ for some
        $L'\in\Pic Y$ and ${\mathcal I}_Z$ being the ideal sheaf of some $0$-dimensional subscheme. Computing
        Chern classes we see that $L'=\OO_Y(-1,-1)$ and $Z=0$.
        Finally, $\Ext 1Y{\OO_Y(-1,-1)}
        {\OO_Y(-1,-1)}=0$ an so we see
        that as an $\OO_Y$-module $M\simeq \OO_Y(-1,-1)\oplus \OO_Y(-1,-1)$.
\end{proof}
This result is very different to what Chan and Kulkarni encountered in \cite{main}. In their example
if an $A$-module was split as an $\OO_Y$-module then they prove that the module must be of the form
$A\otimes_YN$ for some $N\in\Pic Y$. Furthermore, any rank two vector bundle on $Y$ could be
given at most two $A$-module structures. In our case, as the above theorem at least suggests,
the $\OO_Y$-vector bundle $\OO_Y(-1,-1)\oplus\OO_Y(-1,-1)$ can be given an infinite number
of non-isomorphic $A$-module structures. In the following proposition, we prove that this is indeed
the case.
\begin{Prop}\label{Propdim}
        The tangent space  to {\bf Pic}\;A at the point corresponding to $A\otimes_Y\OO_Y(0,-1)$ and
        $A\otimes_Y\OO_Y(-1,0)$ has dimension $1$.
\end{Prop}
\begin{proof}
        The dimension of the tangent space is given by:\\
        $\ext 1A{A\otimes_Y\OO_Y(-1,0)}{A\otimes_Y\OO_Y(-1,0)}\\=\ext 1Y{\OO_Y(-1,0)}
        {\OO_Y(-1,0)\oplus\OO_Y(-1,-2)}=1.
        $ The other case is identical.
\end{proof}

Thus at least one connected component of this moduli space is a smooth curve with all, except at most $2$ points, corresponding
to $A$-modules with the underlying $\OO_Y$-module structure being $\OO_Y(-1,-1)\oplus\OO_Y(-1,-1)$.

We finish off the section with an algebraic description of the $A$-line bundles.
\begin{Prop}\label{PropMinA}
        Let $M\in\Pic A$ with $c_1=\OO_Y(-2,-2)$ and $c_2=2$. Then $\Hom AMA=2$. Further, if
        $0\neq\varphi\in\Hom AMA$ then $\varphi$ is injective.
\end{Prop}
\begin{proof}
        We consider all the possibilities from Theorem \ref{ystruc}.
        If $M\simeq A\otimes_Y\OO_Y(-F)$ then 
        \begin{align*}
                \hom AMA&=\hom A{A\otimes_Y\OO_Y(-F)}A\\
                &=\hom Y{\OO_Y(-F)}{\OO_Y\oplus\OO_Y(-1,-1)}=2.
        \end{align*}

        If, on the other hand, $M\simeq\OO_Y(-1,-1)\oplus\OO_Y(-1,-1)$ as an $\OO_Y$-module then:
        \begin{align*}
                \hom AMA&=\ext 2AA{\omega_A\otimes_AM}^*\\
                &=\ext 2Y{\OO_Y}{\OO_Y(-H)\otimes_YM}^*\\
                &=h^2(Y,\OO_Y(-2,-2)\oplus\OO_Y(-2,-2))\\
                &=h^0(Y,\OO_Y\oplus\OO_Y)=2.
        \end{align*} Since $M$ and $A$ are torsion free, any non zero map $M\to A$ must be injective.
\end{proof}
To understand better how $M$ sits inside $A$ we need to understand the all the possible cokernels. We do so, in the next theorem.
\begin{Thm}\label{quotientscomputation}
        Let $M\in\Pic A$ with $c_1=\OO_Y(-2,-2)$ and $c_2=2$. Then for any $0\neq\varphi\in\hom AMA$
        there exists an exact sequence of $A$-modules
        \[0\longrightarrow M\stackrel\varphi\longrightarrow A\longrightarrow Q\longrightarrow 0\] where:
        \begin{enumerate}
                \item if $M\simeq A\otimes_Y\OO_Y(-1,0)$ (respectively $M\simeq A\otimes_Y\OO_Y(0,-1)$) 
                        then $Q\simeq A\otimes_Y\OO_F$ where $F$ is a $(1,0)$ 
                        (respectively  $(0,1)$) divisor;
                \item if $M\simeq \OO_Y(-1,-1)\oplus\OO_Y(-1,-1)$ 
                        then $Q\simeq \OO_C$ as an $\OO_Y$-module, where $C$ is a $(1,1)$-divisor.
        \end{enumerate}
\end{Thm}
\begin{proof}From the previous proposition, we know $\varphi\!:M\to A$ is injective. 
        Let us compute the cokernel.
        \begin{enumerate}
                \item We prove only the case where $M\simeq A\otimes_Y\OO_Y(-1,0)$ because
                        the other is similar.
                        Note that $\hom Y{\OO_Y(-1,0)}{\OO_Y}=2=\hom AMA$ and
                        so all  $A$-module morphisms arise from an $\OO_Y$-module morphism
                        $\OO_Y(-1,0)\to \OO_Y$ via $A\otimes_Y-$.
                         Since any non zero morphism
                        $\OO_Y(-1,0)\to \OO_Y$ gives rise to the following exact sequence \[0\longrightarrow
                        \OO_Y(-1,0)\longrightarrow \OO_Y\longrightarrow \OO_F\longrightarrow 0\] for some $(1,0)$-divisor
                        $F$
                        and because $A$ is flat over $Y$, the result follows.
                \item Note that with respect to the $\OO_Y$-module decomposition 
                        \begin{align*}
                                M&=\OO_Y(-1,-1)\oplus\OO_Y(-1,-1)\\
                                A&=\OO_Y\oplus\OO_Y(-1,-1)
                        \end{align*} we have
                        \begin{align*}
                                \Hom YMA&=\Hom Y{\OO_Y(-1,-1))\oplus\OO_Y(-1,-1)}{\OO_Y\oplus\OO_Y(-1,-1)}\\
                         &=\left(
                        \begin{array}[]{ll}
                                H^0(Y,\OO_Y(-1,-1)^*)&{\rm End}_Y(\OO_Y(-1,-1))\\
                                H^0(Y,\OO_Y(-1,-1)^*)&{\rm End}_Y(\OO_Y(-1,-1))
                        \end{array}
                        \right).       
                        \end{align*}
                        Thus any $\OO_Y$-module morphism $\varphi\!:M\to A$ is given by 
                        $X=\left( 
                        \begin{array}[]{ll}
                                \varphi_1 &\lambda_1\\
                                \varphi_2 &\lambda_2
                        \end{array}
                        \right)$ where $\varphi_1,\varphi_2\in\OO_Y(-1,-1)^*$ and 
                        $\lambda_1,\lambda_2\in{\rm End}_Y(\OO_Y(-1,-1))=k$ which
                        acts as right multiplication on the row vector $\OO_Y(-1,-1)\oplus\OO_Y(-1,-1)$. For
                        this to be in fact an $A$-module morphism further conditions on $X$ need to be imposed.
                        In particular $\varphi$ needs to be injective and so $\lambda_1,\lambda_2$ are not
                        both zero.
                        
                        We claim that
                        \[Q=\frac{\OO_Y}{ {\rm im}({\lambda_2\varphi_1-\lambda_1\varphi_2})}\] 
                        and that we have the following exact sequence
                        \[0\longrightarrow M\stackrel \varphi\longrightarrow A
                        \stackrel\psi\longrightarrow Q\longrightarrow 0\]
                        with $\psi\!:A\to Q$ given by right multiplication by
                        \[\left\{  
                        \begin{array}[]{cl}
                        \left(\begin{array}[]{c}
                                \lambda_1+\lambda_2\\-(\varphi_1+\varphi_2)
                        \end{array}
                        \right)&\text{if $\lambda_1+\lambda_2\neq 0$}\\
                        \left( 
                        \begin{array}[]{c}
                                \lambda_1\\-\varphi_1
                        \end{array}
                        \right)&\text{if $\lambda_1+\lambda_2=0.$}         
                        \end{array}
                        \right.\]
                        Since $M\to A$ must be injective,
                        ${\rm im}(\lambda_2\varphi_1-\lambda_1 \varphi_2)\neq 0$ and so, $Q$ is isomorphic, 
                        as an $\OO_Y$-module, to $\OO_C$ for some
                        $(1,1)$-divisor $C$. The proof of this claim is just a routine local computation
                        and is done in Lemma 2.4.5 in \cite{Thesis}.
                          \end{enumerate}
\end{proof}
The above theorem suggests that we should study quotients of  $A$. In particular, we should try to better understand
the component(s) of the Hilbert scheme of $A$ containing the $A$-modules whose underlying $\OO_Y$-module structure
is $\OO_C$ where $C$ is a $(1,1)$-divisor. We do this in the following
section.
\section{The Hilbert Scheme of $A$}
In this section we will study the Hilbert scheme of $A$ -- the moduli space
of left sided quotients of $A$ with a fixed set of Chern classes. This is a closed subscheme of
the classical Quot scheme of $A$, which is projective provided we fix a Hilbert polynomial. See Chapter 3 in \cite{Thesis}
for all the details.

Mimicking the commutative case, one should think
of a quotient of $A$, which is supported on a curve on $Y$, as a noncommutative curve
lying on $A$. As mentioned at the end of the last section, we are primarily interested 
in those quotients of $A$ which are supported on a $(1,1)$-divisor on $Y$.
\subsection{Properties of ${\bf Hilb}\;A$}
Recall that in Theorem \ref{quotientscomputation} we saw a link between the moduli space of $A$-line bundles with
with $c_1=\OO_Y(-2,-2)$ and $c_2=2$ and quotients of $A$, or noncommutative curves on $A$, with
$c_1=\OO_Y(1,1)$ and $c_2=2$.

\begin{Prop}\label{PropKer}
       Let $S$ be a scheme. Let $\sheafF$ be a flat family of quotients of $A$ on  $S$ with Chern
       classes $c_1=\OO_Y(1,1)$ and $c_2=2$. Let $I:=\ker(A_S\to \sheafF)$. Then
       $I$ is a flat family of $A$-line bundles on $S$ with Chern classes $c_1=\OO_Y(-2,-2)$ and $c_2=2$.
\end{Prop}
\begin{proof}
  $I$ is flat over $S$ because $A_S$ and $\sheafF$ are. Restricting to the fibre
        above any $p\in S$  we get \[0\longrightarrow I_{k(p)}\longrightarrow A\longrightarrow 
        \sheafF_{k(p)}\longrightarrow 0\] of $A$-modules which is exact because $\sheafF$ is flat
        over $S$ and so $Tor_{\OO_S}^1(\sheafF,k(p))=0$. Since
        \begin{align*}
                c_1(A)=\OO_Y(-1,-1),\quad c_2(A)=0,\quad c_1(\sheafF_{k(p)})=\OO_Y(1,1), \quad c_2(\sheafF_{k(p)})=2       
        \end{align*}
         we see that
         $c_1(I_{k(p)})=\OO_Y(-2,-2)$ and $c_2(I_{k(p)})=2$. 
         $I_{k(p)}$ is torsion free  and so $I_{k(p)}^{**}\in\Pic A$ because
         it is reflexive and hence locally free over $Y$. By
         Proposition \ref{PropDelta} we have $c_2(I_{k(p)}^{**})=2$ and so $I_{k(p)}^{**}=I_{k(p)}$.
        \end{proof}

Having established a relationship between flat families of $A$-line bundles and flat families of
quotients of $A$, we now use Theorem \ref{ystruc} to classify all the possible $\OO_Y$-module structures
that quotients of $A$ may possess. As we shall see some (and, as we shall later see, most)
must all also be quotients of $\OO_Y$.
\begin{Cor}\label{CorQuot}
        Let $Q$ be a quotient of $A$ with $c_1=\OO_Y(1,1)$ and $c_2=2$. Then either:
        \begin{itemize}
                \item $Q\simeq A\otimes_F\OO_F$ (as an $A$-module) where $F$ is either a $(1,0)$ or $(0,1)$-divisor; or
                \item $Q\simeq \OO_C$ (as an $\OO_Y$-module) for some $\sigma$-invariant $(1,1)$-divisor $C\subset Y$.
        \end{itemize}
\end{Cor}
\begin{proof}
        The above proposition asserts that the kernel of $A\to Q$ is an $A$-line bundle with 
        $c_1=\OO_Y(-2,-2)$ and $c_2=2$. We have already classified all such line bundles and their
        respective cokernels in Proposition \ref{ystruc} and Theorem \ref{quotientscomputation}. The fact
        that $C$ must be $\sigma$ invariant
        follows from the fact that in order to be an $A$-module
        there must be a non-zero map $L_\sigma\otimes_Y\OO_C\to\OO_C$
        which is only possible if $\sigma^*C=C$.
\end{proof}

\begin{Cor}\label{smoothsupport}
          Let $Q$ be a quotient of $A$ with $c_1=\OO_Y(1,1)$ and $c_2=2$. If the support $Q$
          is smooth (i.e. its the support is $\P^1$) then $Q$ is also quotient
          of $\OO_Y$.
\end{Cor}
\begin{proof}
        Obvious from the previous Corollary because the support of $\A\OO_F$ is not smooth.
\end{proof}
From now on ${\bf Hilb}\;A$ will denote the Hilbert scheme of $A$
corresponding to quotients of $A$ with $c_1=\OO_Y(1,1)$ and $c_2=2$. We now proceed to study its properties.

\begin{Prop}\label{PropDim}
        The dimension of ${\bf Hilb}\;A$ at the point corresponding to $A\otimes_{Y}\OO_F$, where
        $F$ is a $(1,0)$ or $(0,1)$-divisor is, $2$. 
\end{Prop}
\begin{proof}
        We have 
        \[0\longrightarrow A\otimes_Y\OO_Y(-F)\longrightarrow A\longrightarrow A\otimes_Y\OO_F\longrightarrow 0.\]
        Let $F'=\sigma^* F$. The dimension of the tangent space is given by:
        \begin{align*}
                \hom A{\A \OO_Y(-F)}{\A \OO_F}&=\hom Y{\OO_Y(-F)}{\A \OO_F}\\
                &=\hom Y{\OO_Y(-F)}{\OO_F\oplus\OO_{F'}(-1)}\\
                &=h^0(Y,\OO_F\oplus \OO_{F'})\\
                &=2.
        \end{align*}
\end{proof}
Unfortunately, we were unable to compute the dimension of the tangent space at any other points
as directly as in the above proposition. We thus proceed by first showing that $\Hilb $ is smooth
and later, after a considerable amount of work, that it is connected. This will of course prove
that $\Hilb $ is a smooth projective surface.

\begin{Thm}
        ${\bf Hilb}\;A$ is smooth.
\end{Thm}
\begin{proof}
        Let $Q$ be a quotient of $A$ corresponding to some point $p\in {\bf Hilb}\;A$. Let $M$ the kernel
        of $A\to Q$. We have an exact sequence 
        \[0\longrightarrow M\longrightarrow A\longrightarrow Q\longrightarrow 0
        \tag{\textasteriskcentered}\]
        where by Proposition \ref{PropKer} $M\in \Pic A$. Obstruction to smoothness at $p$ is given by
        $\Ext 1AMQ$ which we now compute. From Corollary \ref{CorQuot} there are only three cases to consider:
        \begin{itemize}
                \item $M\simeq A\otimes_Y\OO_Y(-1,0)$ and $Q\simeq A\otimes_Y\OO_F$ where $F$ is a $(1,0)$
                        divisor. Let $F'=\sigma^*F$ which is a $(0,1)$-divisor.
                        \begin{align*}
                                \ext 1A{A\otimes_Y\OO_Y(-1,0)}{A\otimes_Y\OO_F}&=
                                \ext 1Y{\OO_Y(-1,0)}{\OO_F\oplus\OO_{F'}(-1)}\\
                                &=h^{1}(Y,\OO_F\oplus\OO_{F'})=0.
                        \end{align*}
                \item $M\simeq A\otimes_Y\OO_Y(0,-1)$ and $Q\simeq A\otimes_Y\OO_F$ where $F$ is a $(0,1)$
                        divisor. The proof is the same as in the case above.
                \item $M\simeq \OO_Y(-1,-1)\oplus\OO_Y(-1,-1)$ as an $\OO_Y$-module and $Q\simeq \OO_C$
                        as an $\OO_Y$-module for some $(1,1)$-divisor $C$. Using Serre duality, we have:
                        \[\ext 1A M{\OO_C}=\ext 1A{\OO_C}{\OO_Y(-H)\otimes_{Y}M}.\]
                        Using the local-global spectral sequence we have 
                        \begin{align*}
                        0&\to
                        H^1(Y,\HOM A{\OO_C}{\OO_Y(-H)\otimes_YM})\to\Ext1A{\OO_C}{\OO_Y(-H)\otimes_YM}\\&\to
                        H^0(Y,\EXT 1A{\OO_C}{\OO_Y(-H)\otimes_YM}).
                        \end{align*}
                        $\HOM A{\OO_C}{\OO_Y(-H)\otimes_YM}=0$ since $\OO_C$ is a torsion sheaf.
                        Furthermore, (\textasteriskcentered) is a locally projective
                        $A$-module resolution of $\OO_C$ and so we get 
                        \begin{align*}
                                0&\to\HOM AA{\OO_Y(-H)\otimes_YM} \to \HOM AM{\OO_Y(-H)\otimes_YM}
                                \\&\to \EXT 1A{\OO_C}{\OO_Y(-H)\otimes_YM }\to 0.
                        \end{align*} Finally, since 
                        \[H^0(\HOM AM{\OO_Y(-H)\otimes_Y M})=0\] and
                        \[H^1(\HOM AA{\OO_Y(-H)\otimes_YM})=H^1(Y,\OO_Y(-H)\otimes_YM)=0\]
                        we see that
                        \[H^0(Y,\EXT 1A{\OO_C}{\OO_Y(-H)\otimes_YM})=0\] and so the result follows.

        \end{itemize}
\end{proof}
Thus, so far we know that at least one connected component of ${\bf Hilb}\;A$ is a smooth projective surface.
As mentioned earlier, in the next section
we will see that in fact ${\bf Hilb}\;A$ is connected, which will prove that this must be its only component.

Corollary \ref{CorQuot} says that some quotients of $A$ are in fact also quotients of $\OO_Y$. In particular,
 they are isomorphic to $\OO_C$ where $C$ is a $\sigma$-invariant $(1,1)$-divisor.
 Furthermore,  the support of $\A \OO_F$ is $F\cup \sigma^*F$
 which is also a $\sigma$-invariant $(1,1)$-divisor. Since
 the tangent space at the points corresponding to $\A\OO_F$ is two, whilst  
 dim $|F|=1$ it must be the case that every connected component of $\Hilb$ has
 a dense subset whose points correspond to quotients of $A$ that
 are also quotients of $\OO_Y$. We may thus expect that there is at least a rational map 
 from the Hilbert scheme of $A$ to the Hilbert scheme
of $Y$. We now explore this further. Note first of all,
that all $\sigma$-invariant $(1,1)$-divisors are equal to $\pi^*l$ where $l$ is a line on $Z$. Furthermore,
lines on $Z$ are parameterised by $\dualP\simeq\P^2$. Thus we can view $\dualP$ as the parameter space of 
$\sigma$-invariant $(1,1)$-divisors.

\begin{Thm}\label{Thm:map}Let $\sheafF$ be the universal family of quotients of $A$ on\\ $Y\times_k {\bf Hilb}\;A$.
        There exists a regular map 
        \begin{align*}
                \Psi\!:{\bf Hilb}\;A&\longrightarrow \dualP\\
                p&\longmapsto {\rm supp}\;\sheafF_{k(p)}
        \end{align*}
\end{Thm}

\begin{proof}
        From the surjective morphism $A_{ {\bf Hilb}\;A}\to \sheafF$ we get a morphism
        $\phi\!:\OO_{Y\times{\bf Hilb}\;A}\to \sheafF$. Define $Q:={\rm coker}\;\phi$
        and note that $Q\neq 0$ since at any on point on $\Hilb$ corresponding to quotient of the form  $\A\OO_F$
        the map $\OO_Y\to\A\OO_F$ has a non-trivial cokernel.
        Let $U\subset {\bf Hilb}\;A$ be the
        open subset $\Hilb-\pi_H({\rm supp}\;Q)$ where $\pi_H\!: Y\times\Hilb\to\Hilb$ is the
        natural projection map. Thus $\sheafF|_{Y\times U}$ is a flat family of $\OO_Y$-quotients
        over $U$. Since $Q_{k(p)}=0$ 
        precisely for those $p\in \Hilb$ which correspond to quotients of $A$ which are also quotients
        of $\OO_Y$, from the discussion preceding this theorem, we know that $U$ is dense
        in $\Hilb$.
        Thus we get a  rational map $\Psi\!:\Hilb\dashrightarrow{\bf Hilb}\;Y$. 
        From Corollary \ref{CorQuot} we know that quotients of $A$ which are also
        quotients of $\OO_Y$ are isomorphic as
        $\OO_Y$-modules to $\OO_C$ where $C$ is a $\sigma$-invariant $(1,1)$-divisor. Thus 
        ${\rm im}\;\Psi\subseteq \dualP$. We will see in Lemma \ref{LemmaFinitePsi} that $\Psi$ is 
        in fact finite to one and so each connected component of $\Hilb$ has at most dimension $2$. 
        Thus
        from Chapter 2 Section 3 Theorem 3 of \cite{Shaf} $\Psi$ is not regular at at most
        only a finite 
        number of points. We claim that $\Psi$ is in fact regular everywhere. To see this,
        let $\widetilde\Hilb\to \Hilb$ be the resolution of indeterminacy of $\Psi$. Let $p\in \Hilb-U$ and let
        $B$ be a smooth curve in $\Hilb$ such that $B\cap U=B-p$; i.e. the only point of $B$ not corresponding
        to a quotient of $\OO_Y$ is $p$.
        Denote by $\widetilde B$ its strict transform. Since $B$ is smooth $B\simeq \tilde B$ and so we may 
        compare families over the two curves.
        We have a map $\widetilde B\to {\bf Hilb}\;Y$ and we denote the corresponding flat family over $\widetilde B$
        of $\OO_Y$-quotients
        by $S'$. Let $S:={\rm supp}\;\sheafF|_{Y\times B}$ and note that $S$ is a family of $\OO_Y$-quotients
        on $B$ but we don't know that it is flat over $B$ and so we proceed rather subtly. Note that $S=S'$ on $B\cap U$.
        Proposition 9.8 of Chapter 3 in \cite{Hartshorne} implies that once we have a flat family over $B-p$ then
        there is only one way to complete it to a flat family over $B$ and that is by taking the scheme theoretic
        closure in $Y\times B$. However, $S$ is closed and so $S'\subseteq S$. Since $S'$ is a family of quotients of $\OO_Y$ 
        with $c_1=\OO_Y(1,1)$ and $c_2=2$ it follows that $S'|_p=S|_p$ and so $S'=S$. Thus, regardless of the choice
        of curve $B$, the image of the point $p$ doesn't change. Hence $\widetilde\Hilb=\Hilb$ and so
        $\Psi$ is regular.
\end{proof}
In summary, the map $\Psi$ does the following: every closed point on $\Hilb$ corresponds to some quotient of $A$. There are two
possibilities: either 
\begin{itemize}
        \item[(i)] it is also a quotient of $\OO_Y$, in which case it is isomorphic, and an $\OO_Y$-module, to $\OO_C$ where $C$
        is a $\sigma$-invariant $(1,1)$-divisor, or
        \item[(ii)] it is not a quotient of $\OO_Y$, then it is isomorphic, as an $A$-module,
        to $\A\OO_F$ where $F$ is either a $(0,1)$ or $(1,0)$-divisor.
\end{itemize}
The crucial point is that the support of $\A\OO_F$ is also
a $\sigma$-invariant $(1,1)$-divisor. Thus to every closed point on $\Hilb$ one can associate a $\sigma$-invariant $(1,1)$-divisor. Since
$\sigma$-invariant $(1,1)$-divisors are parameterised by $\dualP$, we get a natural set-theoretic map from (closed points of $\Hilb$) $\to$
(closed points of $\dualP$).
The above theorem proves that this map is in fact morphism of schemes.

\subsection{The ramification of $\Psi\!:\Hilb\to\dualP$}\label{SectionOc}
We want to study the map $\Psi$, in particular we want to understand its
ramification for then we will be able to later compute $(K_{\Hilb})^2$. 
This amounts to computing the number of 
quotients of $A$ which have support a $\sigma$-invariant $(1,1)$-divisor
and $c_2=2$. Corollary \ref{CorQuot} implies that this question will be answered
provided we can understand the number of $A$-module structures that $\OO_C$ can be given, where $C$ is
a $\sigma$-invariant $(1,1)$-divisor.

To give a coherent sheaf  $\sheafG$ on $Y$ an $A$-module structure amounts
to giving a left $\OO_Y$-module morphism $\varphi\!:\A\sheafG\to\sheafG$ satisfying 
the necessary associativity condition. Two such morphisms $\varphi,\varphi'$ give 
rise to isomorphic $A$-modules provided there exists $\psi\in{\rm Aut}_Y\sheafG$
such that\[\xymatrix{
\A\sheafG \ar[r]^-\varphi\ar[d]_{{\rm id}\otimes\psi}\ar@<1.3ex>@{}[d]_(0.6){\rotsymb{\sim}{90}}&\sheafG\ar[d]^\psi_(0.6){\rotsymb{\sim}{90}}\\
\A\sheafG\ar[r]_-{\varphi'}&\sheafG
}\] commutes. In general it may be rather difficult to determine whether such a $\psi$ exists, and consequently, whether two seemingly 
different $A$-module structures are actually isomorphic. The problem becomes
increasingly difficult as the size of ${\rm Aut}_Y\sheafG$ increases. Luckily, in our case, this issue is easily manageable.
\begin{Eg}
        We can illustrate of the above phenomenon with two (related) examples. Recall
        from Theorem \ref{ystruc} that an $A$-line bundle had two possible $\OO_Y$-module structures:
        either it was $\OO_Y(-1,-1)\oplus\OO_Y(-1-1)$ or $\A\OO_Y(-1,0)\stackrel{Y}\simeq\OO_Y(-1,0)
        \oplus\OO_Y(-1,-2)$. The former, as we later saw, had infinitely many non-isomorphic $A$-module
        structures whilst the latter, only had one. The fact that that $\OO_Y(-1,0)
        \oplus\OO_Y(-1,-2)$ has only one $A$-module structure is only clear, when it is written as
        $\A\OO_Y(-1,0)$ which clearly only has one $A$-module structure.
        Hence, if one does not realise that $\OO_Y(-1,0)
        \oplus\OO_Y(-1,-2)\stackrel{Y}\simeq\A\otimes\OO_Y(-1,0)$ then determining the fact
        that all possible $A$-module structures are isomorphic may be very hard indeed. 

        A similar phenomenon occurs for quotients of $A$.
        Let $Q:=\A\OO_F$ and forget the natural $A$-module structure, and ask: how many
        (non-isomorphic) $A$-module structures can $Q$ have? If one does not realise that at least
        as an $\OO_Y$-module $Q\simeq \A\OO_F$ it will be difficult to prove that all the potentially
        different $A$-module structures are in fact isomorphic. Furthermore, as we are about to
        see, for most $\sigma$-invariant $(1,1)$-divisors $C$, $\OO_C$ will have several, but finitely many,
        $A$-module structures. 

        The reason for the difference in the number of $A$-module structures is partly due
        to the size of the endomorphism ring of the modules. In the first example, ${\rm dim}_k{\rm End}_Y(\OO_Y(-1,-1)\oplus
        \OO_Y(-1,-1))=4$ whilst ${\rm dim}_k{\rm End}_Y(\A\OO_Y(-1,0))=5$. A larger automorphism group
        means it is ``easier'' for two $A$-modules structures to be isomorphic.
\end{Eg}

We now study the number of $A$-module structures that $\OO_C$ may possess. For any $p\in\dualP$ we will denote
 by $l_p$ the corresponding line in $\P^2$ and we let $C_p:=\pi^*l_p$ which is a $\sigma$-invariant
$(1,1)$-divisor.

As we saw, for every $p\in \dualP$, giving $\OO_{C_p}$ an $A$-module structure amounts to giving a left $\OO_Y$-module
map $\A\OO_{C_p}\to\OO_{C_p}$ satisfying the necessary associativity condition. In order to better understand this
we first introduce some notation: we let $\bar L:=
L\otimes_Y\OO_{C_p}=L|_{C_p}$ and $\bar D:=D\cap C_p$.
Then, since
$\A\OO_{C_p}=A|_{C_p}$ and because $L_\sigma^{\otimes 2}\simeq\OO_Y(-D)$
this condition is equivalent to giving a map $m\!:\bar L\to\OO_{C_p}$  
such that \[\OO_{C_p}(-\bar D)\simeq \bar L_\sigma\otimes_{C_p}\bar L_\sigma
\stackrel{1\otimes m}\longrightarrow \bar L_\sigma\stackrel{m}\longrightarrow \OO_{C_p}(-\bar D)\]
is the identity. Note that given such a map $m$, the map $-m$ gives a different, non
isomorphic $A$-module structure to $\OO_{C_p}$. This observation gives us the following:
\begin{Prop}\label{PropInv}
        There exist an involution $\tau\!:\Hilb\to\Hilb$ sending an $A$-module structure given by 
        $m$ to the one given by $-m$ . The fixed points are those which corresponds
        to quotients of $A$ that are not quotients of $\OO_Y$.
\end{Prop}
\begin{proof}
        If $\tau$ sends the $A$-module structure given by $m$ to the one given by $-m$ then 
        if the module is also
        a quotient of $\OO_Y$ then as we just saw, these two $A$-module structures are not isomorphic. If
        the module is not a quotient of $\OO_Y$ then by Corollary \ref{CorQuot}
        it must be isomorphic to $\A\OO_F$ which can only be given one $A$-module structure.
\end{proof}
\begin{Cor}\label{CorFactor}
        The map $\Psi\!:\Hilb\to\dualP$ factors through $\Hilb/\gen\tau$. I.e. we have the following 
        commutative diagram\[\xymatrix{\Hilb\ar[dd]\ar[dr]\\
        &\Hilb/\gen{\tau}\ar[dl]\\ \dualP}\]
\end{Cor}
\begin{proof}
        Clear from the above proposition and Theorem \ref{Thm:map}.
\end{proof}
We can view $m$ as an element of $H^0(C_p,\bar L^{-1})$
and, up to multiplication by $\pm 1$, 
the associativity condition then simply says that we need ${\rm div}\;m+{\rm div}\;\sigma^*m=\bar D$, where
each such $m$ gives rise to two $A$-module structures.
Since $\bar D$ is a finite number of points we have proved the following lemma, which also finishes
off the proof of the Theorem \ref{Thm:map}:
\begin{Lemma}\label{LemmaFinitePsi}
        The map $\Psi$ is finite.
\end{Lemma}

This
way of thinking, allows us to view the problem of giving $\OO_{C_p}$ an $A$-module structure geometrically.
As we are about to see, the number of $A$-module structures that $\OO_{C_p}$ can be given depends
primarily how many points $l_p$ intersects with $E$ and $E'$.

Note
also that the dual of a smooth conic in $\P^2$ is another smooth conic in $(\P^2)^\vee$. We denote the duals 
of $E$ and $E'$ by $E^\vee$ and $E'^\vee$ respectively. The picture one should keep in
mind is this:

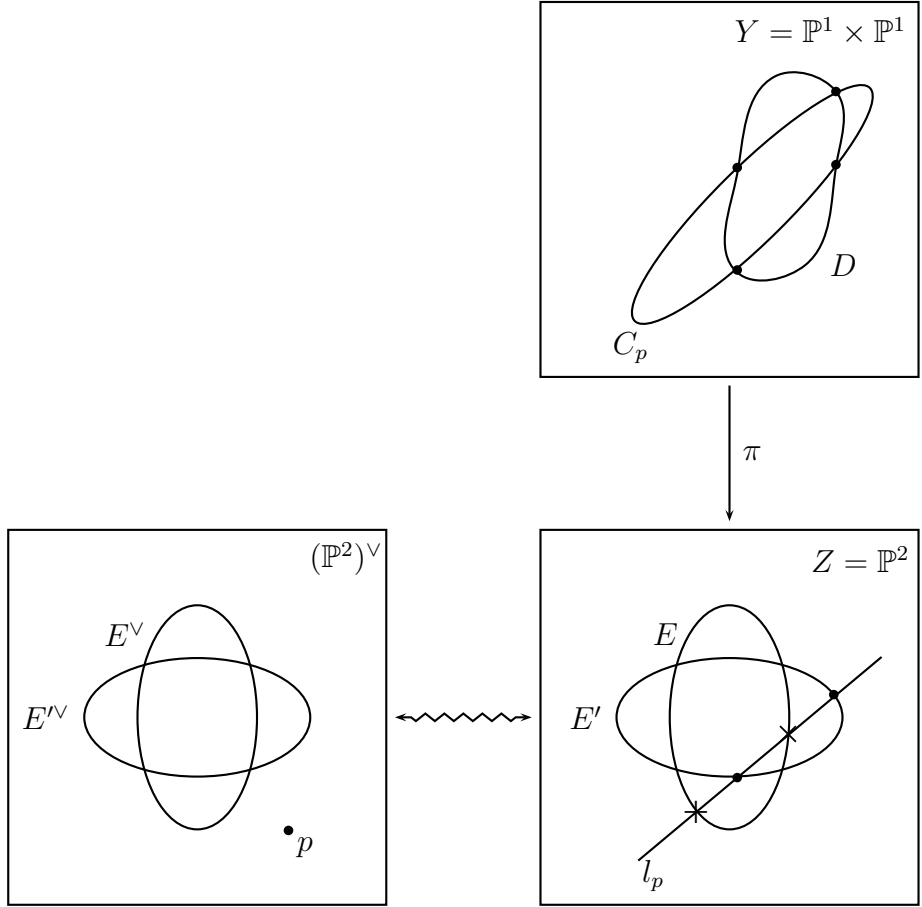
\begin{pspicture}(0,0)(12,12)
\psset{xunit=1cm,yunit=1cm}
\psellipse(2.5,2.5)(0.8,1.5)
\psellipse(2.5,2.5)(1.5,0.8)
\psframe(0,0)(5,5)

\uput[l](2,3.6){$E^\vee$} 
\uput[l](1,2.5){$E'^\vee$} 
\uput[r](3.8,4.6){$(\P^2)^\vee$} 

\psdot(3.7,1)

\psframe(7,0)(12,5)
\psellipse(9.5,2.5)(0.8,1.5)
\psellipse(9.5,2.5)(1.5,0.8)
\uput[l](9,3.6){$E$} 
\uput[l](8,2.5){$E'$} 
\uput[r](10.4,4.6){$Z=\P^2$} 
\psline(11.5,3.3)(8.3,0.6)

\pszigzag[coilwidth=0.1cm,coilheight=2.5,coilarm=0.17cm]{<->}(5.1,2.5)(6.9,2.5)
\psframe(7.0,7.0)(12.0,12.0)
\psline{->}(9.5,6.9)(9.5,5.1)
\uput[r](9.5,6){$\pi$}
\rput{315}(9.8,9.3){
\psellipse(0,0)(0.5,2.2)
}
\rput(8.5,0.4){$l_p$}
\rput(3.92,0.78){$p$}
\psccurve(10.9,9.8)(10.9,10.8)(10.1,11)(9.6,9.7)(9.6,8.4)(10.5,8.5)
\rput(8.2,7.4){$C_p$}       
\rput(11,8.5){$D$}       
\rput(10.7,11.6){$Y=\PP$}       
\psdots(10.9,9.83)(10.9,10.8)(9.6,9.79)(9.6,8.43)
\psdots(9.6,1.7)(10.87,2.8)
\psdots[dotstyle=+, dotscale=2](9.06,1.25)
\psdots[dotstyle=+, dotscale=2, dotangle=45](10.27,2.28)
\end{pspicture}

We mark where $l_p$ intersects $E$ with a
``$\btimes$''
and where $l_p$ intersects $E'$ with a ``$\bullet$''.
The problem  of giving $\OO_{C_p}$ an $A$-module structure
breaks up into two cases:
\begin{enumerate}\label{2cases}
        \item  $l_p$ is not tangential to $E$. In this case we 
                get $C_p\to l_p$ is a $2:1$ cover ramified at two points an hence $C_p\simeq \P^1$,
                in particular it is smooth.
                We analyse this case first, in Section \ref{sec1}.
        \item  $l_p$ is tangential to $E$. In this case  $C_p\to l_p$ is ramified at only one point and 
                hence $C_p$ is the union of two $\P^1$'s, in particular
                it is singular. We analyse this case second, in Section
                \ref{sec2}.
\end{enumerate}

From now on, in any subsequent diagrams, any vertical conic on $Z$ will be $E$, any horizontal one
will be $E'$ and similarly with $E^\vee$ and $E'^\vee$ on $\dualP$ and hence will not longer be labelled.

\subsubsection{If $C$ is smooth}\label{sec1}
As mentioned earlier, we begin by studying the first of the two cases mentioned above. 
Recall that $C_p$ is smooth, in fact $C_p\simeq \P^1$, precisely when $l_p$ is not tangential to $E$
or, equivalently, when $p$ doesn't lie on $E^\vee$.
In this case, from Corollary \ref{smoothsupport} we know
that all quotients of $A$ with this support have their underlying $\OO_Y$-module structure isomorphic
to $\OO_C$. 

This happens when 
$l_p$ is not a tangent to $E$ which is equivalent to $p$ not lying on $E^\vee$.
In this case, since $\Pic C_p\simeq\Z$ we have $H^0(C_p,\bar L^{-1})=H^0(C_p,\OO_{C_p}(2))$
and so to give $\OO_{C_p}$ an $A$-module
structure corresponds to choosing two points  $\bar D'\subseteq \bar D:=C_p\cap D$ such that $\bar D'+ \sigma^*
\bar D'=\bar D$. As mentioned earlier, any such choice gives rise to precisely two $A$-module structures. There
are several cases that need to be considered depending on precisely where $p$ lies.
\begin{enumerate}
        \item[Case 1:] $p$ does not lie
                on either $E^\vee$ or $E'^\vee$ nor on any of the four bitangents to them and so
                we see that this is the generic case. In summary we have:
                \begin{center}
                        \psset{xunit=0.7cm,yunit=0.7cm}
                \begin{tabular}{c|c|c}
                        position of $p\in \dualP$ & position of $l_p\subset \P^2$ & $C_p\to l_p$ \\ \hline
                        \begin{pspicture}(0,0)(5,5.4)
                        \psellipse(2.5,2.5)(0.8,1.5)
                        \psellipse(2.5,2.5)(1.5,0.8)
                        \psframe(0,0)(5,5)
                        \qline(2,0.3)(4.7,3.0) 
                        \qline(2,4.7)(4.7,2.0) 
                        \qline(0.3,2)(3.0,4.7) 
                        \qline(0.3,3)(3,0.3) 
                        \psdot(3.7,1)
                        \rput(3.92,0.78){$p$}
                \end{pspicture}&
                \begin{pspicture}(0,0)(5,5)
                \psellipse(2.5,2.5)(0.8,1.5)
                \psellipse(2.5,2.5)(1.5,0.8)
                \psframe(0,0)(5,5)
                \psline(4.5,3.3)(1.3,0.6)
                \psdots[dotstyle=+, dotscale=2](2.06,1.25)
                \psdots[dotstyle=+, dotscale=2, dotangle=45](3.27,2.28)
                \rput(1.7,0.4){$l_p$}
                \psdots(2.6,1.7)(3.87,2.8)
                \end{pspicture}&
        \begin{pspicture}(0,0)(5,5)
                \psellipse(2.5,3.5)(1.5,0.8) 
                 \psline(0.5,1)(4.5,1)
                 \psline{->}(2.5,2.3)(2.5,1.2)
                 \psdots[dotstyle=+, dotscale=2, dotangle=45](1,1)(4,1)
                 \psdots(2,1)(3,1)(2,2.77)(3,2.77)(2,4.22)(3,4.22)
                \end{pspicture}
\end{tabular}
\end{center}
Thus there are $4$ choices for $\bar D'$ which results in $8$ different $A$-module 
structures on $\OO_{C_p}$.
In order for us to later study the ramification of $\Psi$ we also include the column which shows
which branch corresponds to which module structure.
\begin{center}
        \setstretch{0}
        \begin{tabular}[]{p{3.8cm}|p{3.8cm}|p{3.8cm}}
                \begin{center}
                \begin{spacing}{1}
                        $\bar D'$
                \end{spacing}
                \end{center}&
                
                \begin{center}
                \begin{spacing}{1}
                        Branches above $p$ corresponding to $\bar D'$
                \end{spacing}
                \end{center}
                
                &\begin{center}
                        \begin{spacing}{1}
                                
                 No. of $A$-quotients with support $C_p$
                        \end{spacing}
                \end{center}\\ \hline         
      \psset{xunit=0.7cm,yunit=0.7cm}
      \hspace{0.4cm}\begin{pspicture}(0,0)(3.5,3.3)
                \psellipse(2,1.5)(1.5,0.8) 
                 \psdots(1.5,2.22)(2.5,2.22)
                \end{pspicture}&
      \psset{xunit=0.7cm,yunit=0.7cm}
                \begin{pspicture}(0,0)(3.5,2.6)
                        \psline(0.5,1)(3.5,1)
                           \psline(0.5,2)(3.5,2)
                           \rput(4,2){$1a$}
                           \rput(4,1){$1b$}
                \end{pspicture}
                &\\

                \psset{xunit=0.7cm,yunit=0.7cm}
                 \hspace{0.4cm}\begin{pspicture}(0,0)(3.5,2.6)
                \psellipse(2,1.5)(1.5,0.8) 
                 \psdots(1.5,0.77)(2.5,0.77)
                \end{pspicture}
                
                & \psset{xunit=0.7cm,yunit=0.7cm}
                
                \psset{xunit=0.7cm,yunit=0.7cm}
                \begin{pspicture}(0,0)(3.5,2.6)
                        \psline(0.5,1)(3.5,1)
                           \psline(0.5,2)(3.5,2)
                           \rput(4,2){$2a$}
                           \rput(4,1){$2b$}
                \end{pspicture}
                &\multirow{2}{*}{\hspace{1.7cm} 8}
                \\

\psset{xunit=0.7cm,yunit=0.7cm}
                 \hspace{0.4cm}\begin{pspicture}(0,0)(3.5,2.6)

                \psellipse(2,1.5)(1.5,0.8) 
                 \psdots(1.5,2.22)(2.5,0.77)
                \end{pspicture}

                & \psset{xunit=0.7cm,yunit=0.7cm}
      \psset{xunit=0.7cm,yunit=0.7cm}
                \begin{pspicture}(0,0)(3.5,2.6)
                        \psline(0.5,1)(3.5,1)
                           \psline(0.5,2)(3.5,2)
                           \rput(4,2){$3a$}
                           \rput(4,1){$3b$}
                \end{pspicture}\\

\psset{xunit=0.7cm,yunit=0.7cm}
                \hspace{0.4cm}\begin{pspicture}(0,0)(3.5,2.6)

                \psellipse(2,1.5)(1.5,0.8) 
                 \psdots(1.5,0.77)(2.5,2.22)
                \end{pspicture}& \psset{xunit=0.7cm,yunit=0.7cm}
      \psset{xunit=0.7cm,yunit=0.7cm}
                \begin{pspicture}(0,0)(3.5,2.6)
                        \psline(0.5,1)(3.5,1)
                           \psline(0.5,2)(3.5,2)
                           \rput(4,2){$4a$}
                           \rput(4,1){$4b$}
                \end{pspicture}
\end{tabular}
\end{center} We may thus conclude that $\Psi$ is an $8:1$ cover of $\dualP$. The other cases are used to study the ramification
of this map.
\item [Case 2:] $p$ lies
                on $E'^\vee$ but not on $E^\vee$ nor on any of the four bitangents.

                \begin{center}
                        
                        \psset{xunit=0.7cm,yunit=0.7cm}
                \begin{tabular}{c|c|c}
                        position of $p\in \dualP$ & position of $l_p\subset \P^2$ & $C_p\to l_p$  \\ \hline
                        \begin{pspicture}(0,0)(5,5.4)
                        \psellipse(2.5,2.5)(0.8,1.5)
                        \psellipse(2.5,2.5)(1.5,0.8)
                        \psframe(0,0)(5,5)
                        \qline(2,0.3)(4.7,3.0) 
                        \qline(2,4.7)(4.7,2.0) 
                        \qline(0.3,2)(3.0,4.7) 
                        \qline(0.3,3)(3,0.3) 
                        \psdot(2.5,1.7)
                        \rput(2.5,2.1){$p$}
                \end{pspicture}&
                \begin{pspicture}(0,0)(5,5)
                        \psellipse(2.5,2.5)(0.8,1.5)
                \psellipse(2.5,2.5)(1.5,0.8)
                \psframe(0,0)(5,5)
                \psline(4.5,2.63)(1.3,0.2)
                \psdots[dotstyle=+, dotscale=2](2.4,1.02)
                \psdots[dotstyle=+, dotscale=2](3.09,1.55)
                \rput(1.5,0.8){$l_p$}
                \psdots(3.62,2.)
                \end{pspicture}&
        \begin{pspicture}(0,0)(5,5)
                \psellipse(2.5,3.5)(1.5,0.8) 
                 \psline(0.5,1)(4.5,1)
                 \psline{->}(2.5,2.3)(2.5,1.2)
                 \psdots[dotstyle=+, dotscale=2, dotangle=45](1,1)(4,1)
                 \psdots(2.5,1)(2.5,2.64)(2.5,2.8)(2.5,4.2)(2.5,4.36)
                \end{pspicture}
\end{tabular}

                \end{center}
There are now only $3$ choices for $\bar D'$ as we see in the table below.
\begin{center}
        \setstretch{0}
        \begin{tabular}[]{p{3.8cm}|p{3.8cm}|p{3.8cm}}
                \begin{center}
                \begin{spacing}{1}
                        $\bar D'$
                \end{spacing}
                \end{center}&
                
                \begin{center}
                \begin{spacing}{1}
                        Branches above $p$ corresponding to $\bar D'$
                \end{spacing}
                \end{center}
                
                &\begin{center}
                        \begin{spacing}{1}
                   No. of $A$-quotients with support $C_p$              
                        \end{spacing}
                \end{center}\\ \hline         
             
      \psset{xunit=0.7cm,yunit=0.7cm}
      \hspace{0.4cm}\begin{pspicture}(0,0)(3.5,3.3)
                \psellipse(2,1.5)(1.5,0.8) 
                 \psdots(2,2.33)(2,2.18)
                \end{pspicture}&

      \psset{xunit=0.7cm,yunit=0.7cm}
                \begin{pspicture}(0,0)(3.5,2.6)
                        \psline(0.5,1)(3.5,1)
                           \psline(0.5,2)(3.5,2)
                           \rput(4,2){$1a$}
                           \rput(4,1){$1b$}
                \end{pspicture}
                                \\
                \psset{xunit=0.7cm,yunit=0.7cm}
                \hspace{0.4cm}\begin{pspicture}(0,0)(3.5,2.6)
                \psellipse(2,1.5)(1.5,0.8) 
                 \psdots(2,0.8)(2,0.67)
                \end{pspicture}
                & \psset{xunit=0.7cm,yunit=0.7cm}
      \psset{xunit=0.7cm,yunit=0.7cm}
                \begin{pspicture}(0,0)(3.5,2.6)
                        \psline(0.5,1)(3.5,1)
                           \psline(0.5,2)(3.5,2)
                           \rput(4,2){$2a$}
                           \rput(4,1){$2b$}
                \end{pspicture}
                &
                \hspace{1.7cm}6
                \\

\psset{xunit=0.7cm,yunit=0.7cm}
\hspace{0.4cm}\begin{pspicture}(0,0)(3.5,2.6)
              \rput(0,-1.4){  \psellipse(2,1.5)(1.5,0.8) 
                 \psdots(2,0.73)(2,2.28)}
         \end{pspicture}
                & 
                \psset{xunit=0.7cm,yunit=0.7cm}
      \psset{xunit=0.7cm,yunit=0.7cm}
                \begin{pspicture}(0,0)(3.5,2.6)
                        \psline(1,1)(3.5,1)
                           \psline(1,2)(3.5,2)
                           \psarc(1,1.5){0.35}{90}{270}
                           \rput(4,2){$3a$}
                           \rput(4,1){$4a$}
                \end{pspicture}\\
                &  \psset{xunit=0.7cm,yunit=0.7cm}

         \begin{pspicture}(0,0)(3.5,2.6)
                        \psline(1,1)(3.5,1)
                           \psline(1,2)(3.5,2)
                           \psarc(1,1.5){0.35}{90}{270}
                           \rput(4,2){$3b$}
                           \rput(4,1){$4b$}
                \end{pspicture}
                \end{tabular}
        \end{center}
\item [Case 3:] $p$ lies
                on one exactly one of the four bitangents but not where they meet the conics.

                \begin{center}
                        
                                        \psset{xunit=0.7cm,yunit=0.7cm}
                \begin{tabular}{c|c|c}
                        position of $p\in \dualP$ & position of $l_p\subset \P^2$ & $C_p\to l_p$  \\ \hline
                        \begin{pspicture}(0,0)(5,5.4)
                        \psellipse(2.5,2.5)(0.8,1.5)
                        \psellipse(2.5,2.5)(1.5,0.8)
                        \psframe(0,0)(5,5)
                        \qline(2,0.3)(4.7,3.0) 
                        \qline(2,4.7)(4.7,2.0) 
                        \qline(0.3,2)(3.0,4.7) 
                        \qline(0.3,3)(3,0.3) 
                        \psdot(3.3,1.6)
                        \rput(3.6,1.3){$p$}
                \end{pspicture}&
                \begin{pspicture}(0,0)(5,5)
                        \psellipse(2.5,2.5)(0.8,1.5)
                \psellipse(2.5,2.5)(1.5,0.8)
                \psframe(0,0)(5,5)
                \psline(4,4.45)(1.3,0.2)
                \psdots[dotstyle=+, dotscale=2, dotangle=0](2,1.3)
                \psdots[dotstyle=+, dotscale=2, dotangle=45](3.2,3.2)
                \rput(1.3,0.8){$l_p$}
                \psdots(2.26,1.75)(3.2,3.2)
                \end{pspicture}&
        \begin{pspicture}(0,0)(5,5)
                \psellipse(2.5,3.5)(1.5,0.8) 
                 \psline(0.5,1)(4.5,1)
                 \psline{->}(2.5,2.3)(2.5,1.2)
                 \psdots[dotstyle=+, dotscale=2, dotangle=45](1,1)(4,1)
                 \psdots(2.5,1)(2.5,2.73)(2.5,4.28)(1,1)(1,3.41)(1,3.59)
                \end{pspicture}
\end{tabular}

                \end{center}
There are now only $2$ choices for $\bar D'$ as we explain in the table below.
\begin{center}
        \setstretch{0}
        \begin{tabular}[]{p{3.8cm}|p{3.8cm}|p{3.8cm}}
                \begin{center}
                \begin{spacing}{1}
                        $\bar D'$
                \end{spacing}
                \end{center}&
                
                \begin{center}
                \begin{spacing}{1}
                        Branches above $p$ corresponding to $\bar D'$
                \end{spacing}
                \end{center}
                
                &\begin{center}
                        \begin{spacing}{1}
                   No. of $A$-quotients with support $C_p$              
                        \end{spacing}
                \end{center}\\ \hline

      \psset{xunit=0.7cm,yunit=0.7cm}
      \hspace{0.4cm}\begin{pspicture}(0,0)(3.5,3.3)
               \rput(0,-1.4){ \psellipse(2,1.5)(1.5,0.8) 
               \psdots(2,2.28)(0.54,1.5)}
                 \end{pspicture}&
      \psset{xunit=0.7cm,yunit=0.7cm}
                \begin{pspicture}(0,0)(3.5,2.6)
                        \psline(1,1)(3.5,1)
                           \psline(1,2)(3.5,2)\psarc(1,1.5){0.35}{90}{270}
                           \rput(4,2){$1a$}
                           \rput(4,1){$4a$}
                \end{pspicture}
                \\
                &
                 \psset{xunit=0.7cm,yunit=0.7cm}
      \psset{xunit=0.7cm,yunit=0.7cm}
                \begin{pspicture}(0,0)(3.5,2.6)
                        \psline(1,1)(3.5,1)
                           \psline(1,2)(3.5,2)\psarc(1,1.5){0.35}{90}{270}
                           \rput(4,2){$1b$}
                           \rput(4,1){$4b$}
                \end{pspicture}&
                
                \hspace{1.7cm}4\\
\psset{xunit=0.7cm,yunit=0.7cm}
\hspace{0.4cm}\begin{pspicture}(0,0)(3.5,2.6)
             \rput(0,-1.4){   \psellipse(2,1.5)(1.5,0.8) 
                 \psdots(2,0.74)(0.54,1.5)}
         \end{pspicture}&
\psset{xunit=0.7cm,yunit=0.7cm}
                 
                \psset{xunit=0.7cm,yunit=0.7cm}
      \psset{xunit=0.7cm,yunit=0.7cm}
                \begin{pspicture}(0,0)(3.5,2.6)
                        \psline(1,1)(3.5,1)
                           \psline(1,2)(3.5,2)
                           \psarc(1,1.5){0.35}{90}{270}
                           \rput(4,2){$2a$}
                           \rput(4,1){$3a$}
                \end{pspicture}\\
                &  \psset{xunit=0.7cm,yunit=0.7cm}

         \begin{pspicture}(0,0)(3.5,2.6)
                        \psline(1,1)(3.5,1)
                           \psline(1,2)(3.5,2)
                           \psarc(1,1.5){0.35}{90}{270}
                           \rput(4,2){$2b$}
                           \rput(4,1){$3b$}
                \end{pspicture}
                \end{tabular}
\end{center}
\item [Case 4:]$p$ 
                is chosen to be the point of intersection of two bitangents to $E^\vee$ and
                $E'^\vee$.
                \begin{center}
                        
                                        \psset{xunit=0.7cm,yunit=0.7cm}
                \begin{tabular}{c|c|c}
                        position of $p\in \dualP$ & position of $l_p\subset \P^2$ & $C_p\to l_p$  \\ \hline
                        \begin{pspicture}(0,0)(5,5.4)
                        \psellipse(2.5,2.5)(0.8,1.5)
                        \psellipse(2.5,2.5)(1.5,0.8)
                        \psframe(0,0)(5,5)
                        \qline(2,0.3)(4.7,3.0) 
                        \qline(2,4.7)(4.7,2.0) 
                        \qline(0.3,2)(3.0,4.7) 
                        \qline(0.3,3)(3,0.3) 
                        \psdot(4.2,2.5)
                        \rput(4.2,2){$p$}
                \end{pspicture}&
                \begin{pspicture}(0,0)(5,5)
                        \psellipse(2.5,2.5)(0.8,1.5)
                \psellipse(2.5,2.5)(1.5,0.8)
                \psframe(0,0)(5,5)
                \psline(4,4)(1,1)
                \psdots[dotstyle=+, dotscale=2, dotangle=20](1.8,1.8)
                \psdots[dotstyle=+, dotscale=2, dotangle=45](3.2,3.2)
                \rput(1.3,0.8){$l_p$}
                \psdots(1.8,1.8)(3.2,3.2)
                \end{pspicture}&
        \begin{pspicture}(0,0)(5,5)
                \psellipse(2.5,3.5)(1.5,0.8) 
                 \psline(0.5,1)(4.5,1)
                 \psline{->}(2.5,2.3)(2.5,1.2)
                 \psdots[dotstyle=+, dotscale=2, dotangle=45](1,1)(4,1)
                 \psdots(4,1)(4,3.41)(4,3.59)(1,1)(1,3.41)(1,3.59)
                \end{pspicture}
\end{tabular}
                \end{center}
There is now only $1$ choice for $\bar D'$ as we explain in the table below.
\begin{center}
         \setstretch{0}
        \begin{tabular}[]{p{3.8cm}|p{3.8cm}|p{3.8cm}}
                \begin{center}
                \begin{spacing}{1}
                        $\bar D'$
                \end{spacing}
                \end{center}&
                
                \begin{center}
                \begin{spacing}{1}
                        Branches above $p$ corresponding to $\bar D'$
                \end{spacing}
                \end{center}
                
                &\begin{center}
                        \begin{spacing}{1}
                   No. of $A$-quotients with support $C_p$              
                        \end{spacing}
                \end{center}\\ \hline

      \psset{xunit=0.7cm,yunit=0.7cm}
      \hspace{0.4cm}\begin{pspicture}(0,0)(3.5,3.3)
                \rput(0,-1.4){\psellipse(2,1.5)(1.5,0.8) 
                \psdots(3.46,1.5)(0.53,1.5)}
                 \end{pspicture}&
      \psset{xunit=0.7cm,yunit=0.7cm}
                \begin{pspicture}(0,0)(3.5,2.6)
                        \psline(1,1)(3.5,1)
                        \psline(1,1.5)(3.5,1.5)
                        \psline(1.15,2)(3.5,2)
                           \psline(1.15,0.5)(3.5,0.5)
                           \psarc(1.29,1.25){0.53}{-180}{-100}
                           \psarc(1.29,1.25){0.53}{100}{180}
                           \psarc(1,1){0.35}{90}{150}
                           \psarc(1,1.5){0.355}{-155}{-90}
                           \rput(4,2){$1a$}
                           \rput(4,1.5){$2a$}
                           \rput(4,1){$3a$}
                           \rput(4,0.5){$4a$}
                \end{pspicture}
                &\hspace{1.7cm}{2}\\
                &
                \psset{xunit=0.7cm,yunit=0.7cm}
                \begin{pspicture}(0,0)(3.5,2.6)
                        \psline(1,1)(3.5,1)
                        \psline(1,1.5)(3.5,1.5)
                        \psline(1.15,2)(3.5,2)
                           \psline(1.15,0.5)(3.5,0.5)
                           \psarc(1.29,1.25){0.53}{-180}{-100}
                           \psarc(1.29,1.25){0.53}{100}{180}
                           \psarc(1,1){0.35}{90}{150}
                           \psarc(1,1.5){0.355}{-155}{-90}
                           \rput(4,2){$1b$}
                           \rput(4,1.5){$2b$}
                           \rput(4,1){$3b$}
                           \rput(4,0.5){$4b$}
                           \end{pspicture} 
        \end{tabular}
\end{center}

\item [Case 5:] $p$ lies on the intersection of one of the bitangents and $E'^\vee$.
                 
                 \begin{center}
                         
                                        \psset{xunit=0.7cm,yunit=0.7cm}
                \begin{tabular}{c|c|c}
                        position of $p\in \dualP$ & position of $l_p\subset \P^2$ & $C_p\to l_p$  \\ \hline
                        \begin{pspicture}(0,0)(5,5.4)
                        \psellipse(2.5,2.5)(0.8,1.5)
                        \psellipse(2.5,2.5)(1.5,0.8)
                        \psframe(0,0)(5,5)
                        \qline(2,0.3)(4.7,3.0) 
                        \qline(2,4.7)(4.7,2.0) 
                        \qline(0.3,2)(3.0,4.7) 
                        \qline(0.3,3)(3,0.3) 
                        \psdot(3.84,2.14)
                        \rput(4,1.8){$p$}
                \end{pspicture}&
                \begin{pspicture}(0,0)(5,5)
                        \psellipse(2.5,2.5)(0.8,1.5)
                \psellipse(2.5,2.5)(1.5,0.8)
                \psframe(0,0)(5,5)
                \psline(4.4,2.85)(1,3.85)
                \rput(1.3,4.2){$l_p$}
                \end{pspicture}&
        \begin{pspicture}(0,0)(5,5)
                \psellipse(2.5,3.5)(1.5,0.8) 
                 \psline(0.5,1)(4.5,1)
                 \psline{->}(2.5,2.3)(2.5,1.2)
                 \psdots[dotstyle=+, dotscale=2, dotangle=45](1,1)(4,1)
                 \psdots(1,0.92)(1.08,3.41)(1.08,3.59)(1,1.08)(0.98,3.41)(0.98,3.59)
                \end{pspicture}
\end{tabular}

                 \end{center}
There is now only $1$ choice for $\bar D'$ as we explain in the table below.
\begin{center}
          \setstretch{0}
        \begin{tabular}[]{p{3.8cm}|p{3.8cm}|p{3.8cm}}
                \begin{center}
                \begin{spacing}{1}
                        $\bar D'$
                \end{spacing}
                \end{center}&
                
                \begin{center}
                \begin{spacing}{1}
                        Branches above $p$ corresponding to $\bar D'$
                \end{spacing}
                \end{center}
                
                &\begin{center}
                        \begin{spacing}{1}
                   No. of $A$-quotients with support $C_p$              
                        \end{spacing}
                \end{center}\\ \hline         
       
      \psset{xunit=0.7cm,yunit=0.7cm}
      \hspace{0.4cm}\begin{pspicture}(0,0)(3.5,3.3)
       \rput(0,-1.4){\psellipse(2,1.5)(1.5,0.8) 
       \psdots(0.54,1.41)(.54,1.59)}
                 \end{pspicture}&
      \psset{xunit=0.7cm,yunit=0.7cm}
                \begin{pspicture}(0,0)(3.5,2.6)
                        \psline(1,1)(3.5,1)
                        \psline(1,1.5)(3.5,1.5)
                        \psline(1.15,2)(3.5,2)
                           \psline(1.15,0.5)(3.5,0.5)
                           \psarc(1.29,1.25){0.53}{-180}{-100}
                           \psarc(1.29,1.25){0.53}{100}{180}
                           \psarc(1,1){0.35}{90}{150}
                           \psarc(1,1.5){0.355}{-155}{-90}
                           \rput(4,2){$1a$}
                           \rput(4,1.5){$2a$}
                           \rput(4,1){$3a$}
                           \rput(4,0.5){$4a$}
                \end{pspicture}
                &\hspace{1.7cm}2\\
                &
                \psset{xunit=0.7cm,yunit=0.7cm}
                \begin{pspicture}(0,0)(3.5,2.6)
                        \psline(1,1)(3.5,1)
                        \psline(1,1.5)(3.5,1.5)
                        \psline(1.15,2)(3.5,2)
                           \psline(1.15,0.5)(3.5,0.5)
                           \psarc(1.29,1.25){0.53}{-180}{-100}
                           \psarc(1.29,1.25){0.53}{100}{180}
                           \psarc(1,1){0.35}{90}{150}
                           \psarc(1,1.5){0.355}{-155}{-90}
                           \rput(4,2){$1b$}
                           \rput(4,1.5){$2b$}
                           \rput(4,1){$3b$}
                           \rput(4,0.5){$4b$}
                           \end{pspicture} 
        \end{tabular}
\end{center}
\end{enumerate}
\subsubsection{If $C$ is singular.}\label{sec2}
We now analyse the second case mentioned on page \pageref{2cases}. 
Here $C_p$ is singular, in fact it is the union of two $\P^1$'s
crossing at one point. This occurs precisely when $l_p$ is tangential to $E$
or, equivalently,
when $p$ lies on $E^\vee$. Let $C_p=F_p+ F'_p$ where
$F_p$ is a $(1,0)$-divisor and $F'_p=\sigma^*F_p$ which  is a $(0,1)$-divisor. In this
case $\Pic C_p=\Z\oplus\Z$.
Thus $H^0(C_p,\bar L^{-1})=H^0(C_p,\OO_{C_p}(1,1))$ and so to give $\OO_{C_p}$ an $A$-module
structure corresponds to choosing two points  $\bar D'\subseteq \bar D:=C_p\cap D$
one lying on $F_p$ the other on $F'_p$
such that $\bar D'+\sigma^*
\bar D'=\bar D$. As before, any such choice gives rise to precisely two $A$-module structures. Since we must choose
one point from $F_p$ and the other from $F'_p$ (and can not choose
both points to lie on $F_p$ nor on $F'_p$) implies that
we have ``lost'' some quotients of $A$ corresponding to $p$. From a geometric view point, this means
that $\bar D' =      
                              \psset{xunit=0.3cm,yunit=0.3cm}
         \begin{pspicture}(0,0)(3.5,3.3)
                         \rput(0,-1){
                           \psline(0.5,2.3)(3.5,0.7)
                           \psline(0.5,0.7)(3.5,2.3)
                           \psdots(2.5,1.78)(3,2.04)}
                 \end{pspicture}$ and $\bar D' =      
                              \psset{xunit=0.3cm,yunit=0.3cm}
         \begin{pspicture}(0,0)(3.5,3.3)
                         \rput(0,-1){
                           \psline(0.5,2.3)(3.5,0.7)
                           \psline(0.5,0.7)(3.5,2.3)
                           \psdots(2.5,1.2)(3,1.01)}
                 \end{pspicture}$
                 do not correspond to $A$-module structures on $\OO_{C_p}$.
                 
However we are now
in the case where Corollary \ref{smoothsupport} no longer applies, and so
not all quotients of $A$ have their underlying $\OO_Y$-module structure equal to $\OO_C$ for
some $(1,1)$-divisor $C$. In fact
from Corollary \ref{CorQuot} we know that for every $p$ lying on $E'^\vee$ there are two additional quotients
of $A$ (in the sense
that they have no analogue in Cases 1-5 because
they are not quotients of $\OO_Y$) with support $C_p$ and they are $\A\OO_{F_p}$ and $\A\OO_{F'_p}$. It is thus
natural to think of the above two choices of $\bar D'$ as giving rise to these two quotients of $A$ and
so we make this association in our future analysis of $\Psi$.

\begin{enumerate}
        \item [Case 6:]$p$ lies
                on $E^\vee$ but not on $E'^\vee$ nor on any of the four bitangents.

                \begin{center}
                        
                        \psset{xunit=0.7cm,yunit=0.7cm}
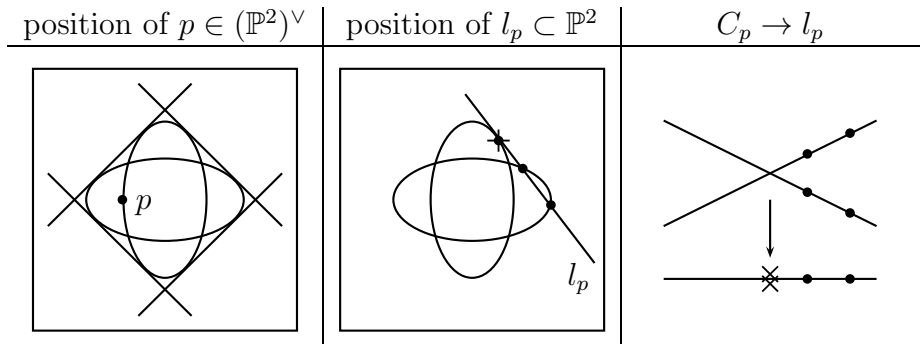
                \begin{tabular}{c|c|c}
                        position of $p\in \dualP$ & position of $l_p\subset \P^2$ & $C_p\to l_p$  \\ \hline
                        \begin{pspicture}(0,0)(5,5.4)
                        \psellipse(2.5,2.5)(0.8,1.5)
                        \psellipse(2.5,2.5)(1.5,0.8)
                        \psframe(0,0)(5,5)
                        \qline(2,0.3)(4.7,3.0) 
                        \qline(2,4.7)(4.7,2.0) 
                        \qline(0.3,2)(3.0,4.7) 
                        \qline(0.3,3)(3,0.3) 
                        \psdot(1.7,2.5)
                        \rput(2.1,2.4){$p$}
                \end{pspicture}&
                \begin{pspicture}(0,0)(5,5)
               \rput(5,0){\rput{90}(0,0){   \psellipse(2.5,2.5)(0.8,1.5)
                \psellipse(2.5,2.5)(1.5,0.8)
                \psframe(0,0)(5,5)
                \psline(4.5,2.63)(1.3,0.2)
                \psdots(2.4,1.02)
                \psdots(3.09,1.55)
                \psdots(3.62,2)
                \psdots[dotstyle=+, dotscale=2](3.62,2)
                }}
                \rput(4.5,1){$l_p$}
        \end{pspicture}&
        \begin{pspicture}(0,0)(5,5)
                \psline(0.5,2)(4.5,4)
                \psline(0.5,4)(4.5,2)
                \psline(0.5,1)(4.5,1)
                 \psline{->}(2.5,2.5)(2.5,1.4)
                 \psdots[dotstyle=+, dotscale=2, dotangle=45](2.5,0.91)(2.5,1.09)
                 \psdots(3.2,1)(4,1)(3.2,2.64)(3.2,3.37)(4,2.25)(4,3.76)
                \end{pspicture}
\end{tabular}

                \end{center}
There are now the full $4$ choices for $\bar D'$, however they only gives rise to six quotients of $A$
as we explain below.
\begin{center}
         \setstretch{0}
        \begin{tabular}[]{p{3.8cm}|p{3.8cm}|p{3.8cm}}
                \begin{center}
                \begin{spacing}{1}
                        $\bar D'$
                \end{spacing}
                \end{center}&
                
                \begin{center}
                \begin{spacing}{1}
                        Branches above $p$ corresponding to $\bar D'$
                \end{spacing}
                \end{center}
                
                &\begin{center}
                        \begin{spacing}{1}
                   No. of $A$-quotients with support $C_p$              
                        \end{spacing}
                \end{center}\\ \hline         

         \psset{xunit=0.7cm,yunit=0.7cm}
         \hspace{0.4cm}\begin{pspicture}(0,0)(3.5,3.3)
                           \psline(0.5,2.3)(3.5,0.7)
                           \psline(0.5,0.7)(3.5,2.3)
                 \psdots(2.5,1.78)(3,2.04)
                \end{pspicture}
                &
   \psset{xunit=0.7cm,yunit=0.7cm}
                \begin{pspicture}(0,0)(3.5,2.6)
                        \psline(1,1)(3.5,1)
                           \psline(1,2)(3.5,2)\psarc(1,1.5){0.35}{90}{270}
                           \rput(4,2){$1a$}
                           \rput(4,1){$1b$}
                \end{pspicture}
                &
             
                \\
 \psset{xunit=0.7cm,yunit=0.7cm}
        \hspace{0.4cm}\begin{pspicture}(0,0)(3.5,2.6)
                           \psline(0.5,2.3)(3.5,0.7)
                           \psline(0.5,0.7)(3.5,2.3)
                 \psdots(2.5,1.24)(3,0.97)
                \end{pspicture}
                &
   \psset{xunit=0.7cm,yunit=0.7cm}
                \begin{pspicture}(0,0)(3.5,2.6)
                        \psline(1,1)(3.5,1)
                           \psline(1,2)(3.5,2)\psarc(1,1.5){0.35}{90}{270}
                           \rput(4,2){$2a$}
                           \rput(4,1){$2b$}
                \end{pspicture}&
                \hspace{1.7cm}6
                \\

\psset{xunit=0.7cm,yunit=0.7cm}
        \hspace{0.4cm}\begin{pspicture}(0,0)(3.5,2.6)
                           \psline(0.5,2.3)(3.5,0.7)
                           \psline(0.5,0.7)(3.5,2.3)
                 \psdots(2.5,1.78)(3,0.97)
                \end{pspicture}
                &
   \psset{xunit=0.7cm,yunit=0.7cm}
                \begin{pspicture}(0,0)(3.5,2.6)
                        \psline(0.5,1)(3.5,1)
                           \psline(0.5,2)(3.5,2)
                           \rput(4,2){$3a$}
                           \rput(4,1){$3b$}
                \end{pspicture}\\

\psset{xunit=0.7cm,yunit=0.7cm}
        \hspace{0.4cm}\begin{pspicture}(0,0)(3.5,2.6)
                           \psline(0.5,2.3)(3.5,0.7)
                           \psline(0.5,0.7)(3.5,2.3)
                 \psdots(2.5,1.24)(3,2.04)
                \end{pspicture}
                &
   \psset{xunit=0.7cm,yunit=0.7cm}
                \begin{pspicture}(0,0)(3.5,2.6)
                        \psline(0.5,1)(3.5,1)
                           \psline(0.5,2)(3.5,2)
                           \rput(4,2){$4a$}
                           \rput(4,1){$4b$}
                \end{pspicture}

                              \end{tabular}

\end{center} 
                              
Let us explain further why branches $1a$ and $1b$ come together
                              here and why this case
                              is different to Case 1. Recall that to picking $\bar D' =      
                              \psset{xunit=0.3cm,yunit=0.3cm}
         \begin{pspicture}(0,0)(3.5,3.3)
                         \rput(0,-1){
                           \psline(0.5,2.3)(3.5,0.7)
                           \psline(0.5,0.7)(3.5,2.3)
                           \psdots(2.5,1.78)(3,2.04)}
                 \end{pspicture}$ and $      
                              \psset{xunit=0.3cm,yunit=0.3cm}
         \begin{pspicture}(0,0)(3.5,3.3)
                         \rput(0,-1){
                           \psline(0.5,2.3)(3.5,0.7)
                           \psline(0.5,0.7)(3.5,2.3)
                           \psdots(2.5,1.2)(3,1.01)}
                 \end{pspicture}$ we associate 
                 not a total of four  $A$-module structure on $\OO_{C_p}$ but the two quotients of $A$ that
                 are not quotients of $\OO_Y$ with support $C_p$, namely $\A\OO_F$ and $\A\OO_{F'}$.
                             We also saw that the involution $\tau$ from Proposition \ref{PropInv}
                fixes points of $\Hilb$ corresponding to $\A\OO_F$ and that by Corollary \ref{CorFactor} the map $\Psi$ factors
                through $\tau$. Hence the branches $1a$ and $1b$ must intersect at precisely points corresponding to $\A\OO_F$.                 
                The same
                argument applies to explain why the branches $2a$ and $2b$ also merge.

\item [Case 7:] $p$ 
lies on the intersection of $E^\vee$ and $E'^\vee$.

\begin{center}
        
        \psset{xunit=0.7cm,yunit=0.7cm}
                \begin{tabular}{c|c|c}
                        position of $p\in \dualP$ & position of $l_p\subset \P^2$ & $C_p\to l_p$  \\ \hline
                        \begin{pspicture}(0,0)(5,5.4)\label{case7}
                        \psellipse(2.5,2.5)(0.8,1.5)
                        \psellipse(2.5,2.5)(1.5,0.8)
                        \psframe(0,0)(5,5)
                        \qline(2,0.3)(4.7,3.0) 
                        \qline(2,4.7)(4.7,2.0) 
                        \qline(0.3,2)(3.0,4.7) 
                        \qline(0.3,3)(3,0.3) 
                        \psdot(3.18,3.18)
                        \rput(2.9,2.9){$p$}
                \end{pspicture}&
                \begin{pspicture}(0,0)(5,5)
            
                        \psellipse(2.5,2.5)(0.8,1.5)
                        \psellipse(2.5,2.5)(1.5,0.8)
                        \psframe(0,0)(5,5)
                        \qline(2,4.7)(4.7,2.0) 
                        \rput(4.5,1.8){$l_p$}
                
                \end{pspicture}&
        \begin{pspicture}(0,0)(5,5)
                \psline(0.5,2)(4.5,4)
                \psline(0.5,4)(4.5,2)
                \psline(0.5,1)(4.5,1)
                 \psline{->}(2.5,2.5)(2.5,1.4)
                 \psdots[dotstyle=+, dotscale=2, dotangle=45](2.5,0.91)(2.5,1.09)
                 \psdots(4,0.92)(4,1.08)(4,2.18)(4,2.32)(4,3.68)(4,3.84)
                \end{pspicture}
\end{tabular}

\end{center}
There are now only $4$ choices for $\bar D'$ as we explain in the table below.
\begin{center}
        
          \setstretch{0}
        \begin{tabular}[]{p{3.8cm}|p{3.8cm}|p{3.8cm}}
                \begin{center}
                \begin{spacing}{1}
                        $\bar D'$
                \end{spacing}
                \end{center}&
                
                \begin{center}
                \begin{spacing}{1}
                        Branches above $p$ corresponding to $\bar D'$
                \end{spacing}
                \end{center}
                
                &\begin{center}
                        \begin{spacing}{1}
                   No. of $A$-quotients with support $C_p$              
                        \end{spacing}
                \end{center}\\ \hline         
        
             \psset{xunit=0.7cm,yunit=0.7cm}
             \hspace{0.4cm}\begin{pspicture}(0,0)(3.5,3.3)
                           \psline(0.5,2.3)(3.5,0.7)
                           \psline(0.5,0.7)(3.5,2.3)
                 \psdots(2.8,2)(2.8,1.84)
                \end{pspicture}
                &
   \psset{xunit=0.7cm,yunit=0.7cm}
               \begin{pspicture}(0,0)(3.5,2.6)
                        \psline(1,1)(3.5,1)
                           \psline(1,2)(3.5,2)\psarc(1,1.5){0.35}{90}{270}
                           \rput(4,2){$1a$}
                           \rput(4,1){$1b$}
                \end{pspicture}
                &
                \\
 \psset{xunit=0.7cm,yunit=0.7cm}
        \hspace{0.4cm}\begin{pspicture}(0,0)(3.5,2.6)
                           \psline(0.5,2.3)(3.5,0.7)
                           \psline(0.5,0.7)(3.5,2.3)
 \psdots(2.8,1.15)(2.8,0.99)

                \end{pspicture}
                &
   \psset{xunit=0.7cm,yunit=0.7cm}
                \begin{pspicture}(0,0)(3.5,2.6)
                        \psline(1,1)(3.5,1)
                           \psline(1,2)(3.5,2)\psarc(1,1.5){0.35}{90}{270}
                           \rput(4,2){$2a$}
                           \rput(4,1){$2b$}
                \end{pspicture}&
                \hspace{1.7cm}4
                \\

\psset{xunit=0.7cm,yunit=0.7cm}
\hspace{0.4cm}\begin{pspicture}(0,0)(3.5,2.6)
                           \rput(0,-1.4){\psline(0.5,2.3)(3.5,0.7)
                           \psline(0.5,0.7)(3.5,2.3)
                           \psdots(2.8,1.1)(2.8,1.94)}
                 \end{pspicture}
                &
   \psset{xunit=0.7cm,yunit=0.7cm}
                \begin{pspicture}(0,0)(3.5,2.6)
                        \psline(1,1)(3.5,1)
                           \psline(1,2)(3.5,2)\psarc(1,1.5){0.35}{90}{270}
                           \rput(4,2){$3a$}
                           \rput(4,1){$4a$}
                \end{pspicture}\\

               &

 \psset{xunit=0.7cm,yunit=0.7cm}
                \begin{pspicture}(0,0)(3.5,2.6)
                        \psline(1,1)(3.5,1)
                           \psline(1,2)(3.5,2)\psarc(1,1.5){0.35}{90}{270}
                           \rput(4,2){$3b$}
                           \rput(4,1){$4b$}
                \end{pspicture}

                              \end{tabular}
\end{center}
\item [Case 8:]$p$ is one of the points of intersection of the bitangents with $E^\vee$.
\begin{center}
        
                        \psset{xunit=0.7cm,yunit=0.7cm}
                \begin{tabular}{c|c|c}
                        position of $p\in \dualP$ & position of $l_p\subset \P^2$ & $C_p\to l_p$  \\ \hline
                        \begin{pspicture}(0,0)(5,5.4)
                        \psellipse(2.5,2.5)(0.8,1.5)
                        \psellipse(2.5,2.5)(1.5,0.8)
                        \psframe(0,0)(5,5)
                        \qline(2,0.3)(4.7,3.0) 
                        \qline(2,4.7)(4.7,2.0) 
                        \qline(0.3,2)(3.0,4.7) 
                        \qline(0.3,3)(3,0.3) 
                        \psdot(2.85,3.85)
                        \rput(2.6,3.6){$p$}
                \end{pspicture}&
                \begin{pspicture}(0,0)(5,5)
            
                        \psellipse(2.5,2.5)(0.8,1.5)
                        \psellipse(2.5,2.5)(1.5,0.8)
                        \psframe(0,0)(5,5)
                        \qline(2.8,4.7)(3.8,1) 
                        \rput(4.2,1.1){$l_p$}
                
                \end{pspicture}&
        \begin{pspicture}(0,0)(5,5)
                \psline(0.5,2)(4.5,4)
                \psline(0.5,4)(4.5,2)
                \psline(0.5,1)(4.5,1)
                 \psline{->}(2.5,2.5)(2.5,1.4)
                 \psdots[dotstyle=+, dotscale=2, dotangle=45](2.5,0.91)(2.5,1.09)
                 \psdots(4,1)(4,2.24)(4,3.76)(2.5,3)(2.5,1)
                \end{pspicture}
\end{tabular}

\end{center}
There are now only $2$ choices for $\bar D'$ as we explain in the table below.
\begin{center}
      
          \setstretch{0}
        \begin{tabular}[]{p{3.8cm}|p{3.8cm}|p{3.8cm}}
                \begin{center}
                \begin{spacing}{1}
                        $\bar D'$
                \end{spacing}
                \end{center}&
                
                \begin{center}
                \begin{spacing}{1}
                        Branches above $p$ corresponding to $\bar D'$
                \end{spacing}
                \end{center}
                
                &\begin{center}
                        \begin{spacing}{1}
                   No. of $A$-quotients with support $C_p$              
                        \end{spacing}
                \end{center}\\ \hline

      \psset{xunit=0.7cm,yunit=0.7cm}
      \hspace{0.4cm}\begin{pspicture}(0,0)(3.5,3.3)
                           \psline(0.5,2.3)(3.5,0.7)
                           \psline(0.5,0.7)(3.5,2.3)
                 \psdots(2,1.5)(2.8,1.94)
                \end{pspicture}
                &\psset{xunit=0.7cm,yunit=0.7cm}
    \begin{pspicture}(0,0)(3.5,2.6)
                        \psline(1,1)(3.5,1)
                        \psline(1,1.5)(3.5,1.5)
                        \psline(1.15,2)(3.5,2)
                           \psline(1.15,0.5)(3.5,0.5)
                           \psarc(1.29,1.25){0.53}{-180}{-100}
                           \psarc(1.29,1.25){0.53}{100}{180}
                           \psarc(1,1){0.35}{90}{150}
                           \psarc(1,1.5){0.355}{-155}{-90}
                           \rput(4,2){$1a$}
                           \rput(4,1.5){$1b$}
                           \rput(4,1){$4a$}
                           \rput(4,0.5){$4b$}
                \end{pspicture}&
                \hspace{1.7cm}{2}
                \\

   \psset{xunit=0.7cm,yunit=0.7cm}
   \hspace{0.4cm}\begin{pspicture}(0,0)(3.5,3.3)
                           \psline(0.5,2.3)(3.5,0.7)
                           \psline(0.5,0.7)(3.5,2.3)
                 \psdots(2,1.5)(2.8,1.1)
                \end{pspicture}
                &\psset{xunit=0.7cm,yunit=0.7cm}
    \begin{pspicture}(0,0)(3.5,2.6)
                        \psline(1,1)(3.5,1)
                        \psline(1,1.5)(3.5,1.5)
                        \psline(1.15,2)(3.5,2)
                           \psline(1.15,0.5)(3.5,0.5)
                           \psarc(1.29,1.25){0.53}{-180}{-100}
                           \psarc(1.29,1.25){0.53}{100}{180}
                           \psarc(1,1){0.35}{90}{150}
                           \psarc(1,1.5){0.355}{-155}{-90}
                           \rput(4,2){$2a$}
                           \rput(4,1.5){$2b$}
                           \rput(4,1){$3a$}
                           \rput(4,0.5){$3b$}
                \end{pspicture}

                              \end{tabular}

                      \end{center}

                      Note that the two $A$-module structures with support $C_p=F+F'$
                      are $\A\OO_F$ and $\A\OO_{F'}$.

              \end{enumerate}
By carefully following which branch connects to which branch we can see that $\Hilb$
is in fact connected
and thus we may conclude that $\Hilb$ is in fact a smooth projective surface.
\subsection{Possible second Chern classes of $A$-line bundles}\label{Chernproof}
In this section we tie up one loose end that we have left from Section \ref{Chernclasses} and prove
the existence of lines bundles with all possible combinations of Chern classes, provided
they satisfy our Bogomolov-type inequality. We continue with  the same notation as before.
\begin{Thm}
        Let $c_1\in \Pic Y$ and $c_2\in \Z$ such that $4c_2-c_1^2\geq -2$. Then
        there exists an $M\in\Pic A$ with these Chern classes.
\end{Thm}
Before we begin the proof, we need the following lemma:
\begin{Lemma}\label{NA}
        Let $C$ be a smooth, $\sigma$-invariant $(1,1)$-divisor on $Y$ and $N\in\Pic C$.
        Endow $\OO_C$ with an $A$-module
        structure, which we saw is always possible from Cases 1-5 previously. Then 
        $N$ inherits an $A$-module structure from $\OO_C$.
\end{Lemma}
\begin{proof}
        We need give an $\OO_Y$-module morphism $A\otimes_C N\to N$ satisfying
        the required associativity condition. Suppose $\psi\!:A\otimes_C
        \OO_C\to\OO_C$ is the morphism
        which gives $\OO_C$ its $A$-module structure. Then $A\otimes_C
        N\to \A \OO_C\otimes_{C}N
        \stackrel{\psi\otimes 1} \longrightarrow\OO_C\otimes_CN\to N$ is the required morphism.
\end{proof}
\begin{proof}[Proof of theorem]
        The discriminant of any rank two vector bundle $M$, 
        defined to be the integer $4c_2(M)-c_1(M)^2$,
        is unchanged by tensoring with a line bundle (see Chapter 12.1 of in \cite{Potier}) and so as we saw
        before we can thus assume $c_1=\OO_Y(-1,-1)$ or $c_1=\OO_Y$. We deal with these two cases separately
        although the proofs will be very similar. Fix for the remainder of the proof a smooth $\sigma$-invariant
        $(1,1)$-divisor $C$ and an $A$-module structure on $\OO_C$.

        We will now construct an $A$-line bundle with $c_1=\OO_Y$ and $c_2=n$ for an arbitrary $n\geq 0$.
        Using Lemma \ref{NA} endow $\OO_C(n+2)$ with an $A$-module structure. Note that 
        \begin{align*}
              \Hom A{\A\OO_Y(1,1)}{\OO_{
        C}(n+2)}&=\Hom Y{\OO_Y(1,1)}{\OO_C(n+2)}  \\
        &=\Hom C{\OO_C}{\OO_C(n)}\\
        &=H^0(\P^1,\OO_{\P^1}(n))\neq 0.
        \end{align*}
        We claim that there is at least
        one morphism 
        $\varphi\!:\A\OO_Y(1,1)\to \OO_C(n+2)$ which is surjective. 
        From the above computation, we see that any $A$-module morphism 
        $\A\OO_Y(1,1)\to\OO_C(n+2)$ arises from an $\OO_Y$-module morphism $\phi\!:\OO_C\to\OO_C(n)$.
        Choose $\phi$ in such a away that ${\rm coker}\;\phi=\oplus k_{p_i}$, where
        $k_{p_i}$ is the skyscraper sheaf at $p_i$, with the $p_i$
        lying in the Azumaya locus of $A$. Then, since $A|_{p}=M_2(k)$,
        when we extend $\phi$
        to a morphism $\varphi\!:\A\OO_Y(1,1)\to\OO_C(n+2)$ we must have ${\rm coker}\;\varphi=0$
        for the simple representations of $M_2(k)$ are all two dimensional. Letting  $M:={\rm ker}\;\varphi$ we have
        \[0\longrightarrow M\longrightarrow\A\OO_Y(1,1)\longrightarrow \OO_C(n+2)\longrightarrow 0.
        \tag{$*$}\]
        It is easy to check that $M\in\Pic A$ with $c_1(M)=\OO_Y$ and $c_2(M)=n$.
       
        Constructing an $A$-line bundle with $c_1=\OO_Y(-1,-1)$ and $c_2=n$ for an arbitrary $n\geq 0$
        is an almost identical process where one finds a surjective morphism $\varphi\!: A\to \OO_C(n)$
        in the same manner as before and then proves that the kernel 
        must be a line bundle. A simple computation shows that this
        kernel has the desired Chern classes.
\end{proof}

\section{The Link}
In this section we establish a link between the moduli space of $A$-line bundles with
$c_1=\OO_Y(-2,-2)$ and $c_2=2$, which as before we denote by ${\bf Pic}\;A$, and
the Hilbert scheme of $A$, which parameterises quotients of $A$ with $c_1=\OO_Y(1,1)$
and $c_2=2$, which as before is denoted by $\Hilb$. In particular we will show that
$\Hilb$ is a ruled surface over ${\bf Pic}\;A$. Thus by using the map
$\Psi$ from the previous section, we will calculate $(K_{\Hilb})^2$, which will allow
us to determine the genus of ${\bf Pic}\;A$.

We have already seen this link between line bundles on $A$ and quotients of $A$. It is summarised
with the following exact sequence \[0\longrightarrow M\longrightarrow A\longrightarrow Q\longrightarrow 0\] where:
\begin{itemize}
        \item $Q\simeq \OO_C$ as an $\OO_Y$-module, which occurs precisely when $M\in \Pic A$ with
                $M\simeq \OO_Y(-1,-1)\oplus\OO_Y(-1,-1)$ as an $\OO_Y$-module, or
        \item $Q\simeq \A \OO_F$, where $F$ is a $(1,0)$ (respectively $(0,1)$) divisor, which occurs 
                precisely when $M\simeq \A\OO_Y(-1,0)$
                (respectively $\A\OO_Y(0,-1)$).
\end{itemize} 
Furthermore, we saw in Proposition \ref{PropMinA} that in both cases $\hom AMA =2$ which
suggests there is a $\P:1$ map $\Hilb\to {\bf Pic}\;A$. We prove this now.
\begin{Thm}
        $\Hilb$ is a ruled surface over ${\bf Pic}\;A$.
\end{Thm}
\begin{proof}
        Let $\sheafF$ be the universal family on $\Hilb$. From Proposition \ref{PropKer}
        ${\rm ker}\;(A_{\Hilb}\to \sheafF)$ is a flat family of $A$-line bundles on
        $\Hilb$ and so we get a map $\Phi\!:\Hilb\to{\bf Pic}\;A$. 
        $\bf M$ being smooth and together with Proposition \ref{Propdim} implies one of its components
        is a curve. However, from the previous section we know that
        $\Hilb$ is a smooth projective surface and thus ${\bf Pic}\;A$ must
        in fact be connected and hence must be a smooth projective curve. It thus suffice to show
        that every fibre of $\Phi$ is isomorphic to $\P^1$ which is clear from Proposition \ref{PropMinA}
\end{proof}
Since $\Hilb$ is a ruled surface over ${\bf Pic}\;A$ we can determine the genus of
${\bf Pic}\;A$ using Corollary 2.11 in Chapter 5 of \cite{Hartshorne} which
states that \[(K_{\Hilb})^2=8(1-g({\bf Pic}\;A)).\] Furthermore, we can determine
$(K_{\Hilb})^2$ using the map $\Psi$.
\begin{Thm}
        The moduli space parameterising $A$-line bundles with $c_1=\OO_Y(-2,-2)$
        and $c_2=2$ is a smooth projective curve of genus $2$.
\end{Thm} 
\begin{proof}
        As discussed above, all that we need to do is compute
        $(K_{\Hilb})^2$. Recall from before that
        we have an $8:1$ map $\Psi\!:\Hilb\to \dualP$.
        Thus using Formula 19 of Section 16 in Chapter 1 of \cite{Barth} we have: \[K_{\Hilb}=\Psi^*K_{\dualP}+R\]
        where $R$ is the ramification divisor on $\Hilb$. 
        
        Let us describe $R$. Looking at Case 2 of Section \ref{sec1} we define $R_1$ and $U_1$ to be the divisors
        such that $\Psi^*E'^\vee=2R_1+U_1$. Similarly looking at Case 6 in Section \ref{sec2} we define $R_2$ and $U_2$ to
        be such that $\Psi^*E^\vee=2R_2+U_2$.        
        Denote by $L_3,\cdots,L_6$ the four bitangents to $E^\vee\cup E'^\vee$. Looking at Case 3
        of Section \ref{sec1} we see that $\Psi^*L_i$ is two divisible and we let $R_i$ be such that
        $\Psi^*L_i=2R_i$. Thus $R=R_1+R_2+\cdots+R_6$.
        
        We now compute $(K_{\Hilb})^2=(\Psi^*(K_{\dualP})+R_1+\cdots+R_6)^2$.
        Throughout this calculation $K$ denotes $K_{\dualP}$.
        \begin{itemize}
                \item $(\Psi^*K)^2=8\cdot (-3)^2=72$
                \item $(\Psi^*K).R_1=K.(\Psi_*R_1)=
        2K.E'=2\cdot(-6)=-12$. Similarly,
                \item $(\Psi^*K).R_2=-12$.
                \item $(\Psi^*K).R_i=K.(\Psi_*R_i)=4K.L_i=4\cdot(-3)=-12$ for $i=3,\cdots,6$.
                \item $R_1.R_2=0$ from Case 7 on page \pageref{case7}.
                \item $R_1.R_i=\frac{1}{2}R_1.(\Psi^*L_i)=\frac{1}{2}(\Psi_*R_1).L_i=\frac{1}{2}2E'.L_i=2$ for 
        $i=3,\cdots,6$. Similarly,
                \item $R_2.R_i=2$ for $i=3,\cdots,6$.
                \item $R_i.R_j=\frac{1}{2}(\Psi^*L_i).R_j=\frac{1}{2}L_i.(\Psi_*R_j)
                        =\frac{1}{2}L_i.4L_j=2L_i.L_j=2$ for $i,j=3,\cdots,6$.
        \end{itemize}
        What remains is to compute $R_1^2$ and $R_2^2$. 

        We can see from Case 2, 5 and 7 that $\Psi|_{R_1}\!:R_1\to E'$ 
        is an \'etale double cover of $E'$. Thus $R_1=R_1'+R_1''$ where
        both $R_1'$ and $R_1''$ have genus zero. We now use the adjunction formula (Proposition
        1.5 in Chapter V of \cite{Hartshorne})
        to compute $R_1'^2$. We have
        \begin{align*}
                -2&=R_1'.(2R_1'+\Psi^*K+R_2+R_3+\cdots+R_6)\\
                &= 2R_1'^2+E'.K+0+4\cdot R_1'.\frac{1}{2}(\Psi^*L_3)\\
                &=2R_1'^2-6+4\cdot 1=2R_1'^2-2.
        \end{align*}Thus $R_1'^2=0$ and an identical computation shows $R_1''^2=0$. Thus
        $R_1^2=0$. 

        The same argument shows $R_2^2=0$ since $R_2\to E$ is also an \'etale double cover.


        Thus
        \begin{align*}
                (K_{\Hilb})^2&=(\Psi^*K)^2+R_1^2
                +\cdots+R^6+\\&+2\left((\Psi^*K).R_1+\cdots
                (\Psi^*K).R_6+\sum R_i.R_j\right)\\
               &=72+0+0+4\cdot 2+2\left( 6\cdot(-12)4\cdot 2+4\cdot 2+6\cdot 2 \right)\\
                &=-8
        \end{align*} and so $g({\bf Pic}\;A)=2$.
\end{proof}
Note that at no stage did we use the fact that $\Hilb$ is ruled in order to calculate $(K_{\Hilb})^2$. In
particular, we didn't use the fact that we knew in advance that $(K_{\Hilb})^2$ is a multiple of eight.
We could have simplified the computation above if we had done so, but it seemed nice to spend the extra work
and get an independent confirmation that fact.

As we saw in the above proof $R_2$ is the union of two $\P^1$'s. These $\P^1$'s are
fibres of $\Phi\!:\Hilb\to {\bf Pic}\;A$ above the two very special points on ${\bf Pic}\;A$ corresponding
to the $A$-line bundles $\A\OO_Y(-1,0)$ and $\A\OO_Y(0,-1)$. Since $R_1$ is also a union of
two $\P^1$'s it would have been nice to find the two $A$-line bundles which they are fibres
of, but unfortunately, we were unable to do so.
\\ \\
\noindent Author: Boris Lerner\\
Email: boris@unsw.edu.au\\
Address: School of Mathematics and Statistics\\
University of New South Wales\\
Sydney, 2052, NSW\\
Australia.

\bibliographystyle{alpha}

\begin{thebibliography}{BPVdV84}

\bibitem[AdJ]{artinjong}
M.~Artin and A.~J. de~Jong.
\newblock {\em Stable orders on surfaces}.



\bibitem[BPVdV84]{Barth}
W.~Barth, C.~Peters, and A.~Van~de Ven.
\newblock {\em Compact complex surfaces}, volume~4 of {\em Ergebnisse der
  Mathematik und ihrer Grenzgebiete (3) [Results in Mathematics and Related
  Areas (3)]}.
\newblock Springer-Verlag, Berlin, 1984.

\bibitem[Cha05]{ncyclic}
Daniel Chan.
\newblock Noncommutative cyclic covers and maximal orders on surfaces.
\newblock {\em Adv. Math.}, 198(2):654--683, 2005.

\bibitem[Cha12]{Lectures}
Daniel Chan.
\newblock {\em Lectures on Orders}.
\newblock 2012.
\newblock Pre-print in progress. In the mean time, these notes can be found on
  Dr. Daniel Chan's website http://web.maths.unsw.edu.au/$\sim$danielch/Lect\_Orders.pdf.


\bibitem[CK03]{bimod}
Daniel Chan and Rajesh~S. Kulkarni.
\newblock del {P}ezzo orders on projective surfaces.
\newblock {\em Adv. Math.}, 173(1):144--177, 2003.

\bibitem[CK11]{main}
Daniel Chan and Rajesh~S. Kulkarni.
\newblock Moduli of bundles on exotic del {P}ezzo orders.
\newblock {\em American Journal of Math.}, 133(1):273--293, 2011.


\bibitem[Fri98]{Friedman}
Robert Friedman.
\newblock {\em Algebraic surfaces and holomorphic vector bundles}.
\newblock Universitext. Springer-Verlag, New York, 1998.


\bibitem[Har77]{Hartshorne}
Robin Hartshorne.
\newblock {\em Algebraic geometry}.
\newblock Springer-Verlag, New York, 1977.
\newblock Graduate Texts in Mathematics, No. 52.


\bibitem[HS05]{Hoffmann}
Norbert Hoffmann and Ulrich Stuhler.
\newblock Moduli schemes of generically simple {A}zumaya modules.
\newblock {\em Doc. Math.}, 10:369--389, 2005.

\bibitem[Lan04]{Langer}
Adrian Langer.
\newblock Semistable sheaves in positive characteristic.
\newblock {\em Ann. of Math. (2)}, 159(1):251--276, 2004.

\bibitem[LP97]{Potier}
J.~Le~Potier.
\newblock {\em Lectures on vector bundles}, volume~54 of {\em Cambridge Studies
  in Advanced Mathematics}.
\newblock Cambridge University Press, Cambridge, 1997.
\newblock Translated by A. Maciocia.


\bibitem[Sha94]{Shaf}
Igor~R. Shafarevich.
\newblock {\em Basic algebraic geometry. 1}.
\newblock Springer-Verlag, Berlin, second edition, 1994.
\newblock Varieties in projective space, Translated from the 1988 Russian
  edition and with notes by Miles Reid.

\bibitem[Ler12]{Thesis}
        Boris Lerner.
        \newblock{\em Line Bundles and Curves on a del Pezzo Order}
        \newblock{PhD Thesis, UNSW.}
        \newblock{Available from http://web.maths.unsw.edu.au/$\sim$danielch/thesis/borisPhD.pdf}



\end{thebibliography}

\end{document}